\documentclass[12pt]{iopart}

\usepackage{epstopdf}
\usepackage{subfigure}

\usepackage{latexsym,bm}
\usepackage{xcolor,graphicx}
\usepackage{eso-pic}
\usepackage{caption}
\usepackage{subfigure}
\newtheorem{lemma}{Lemma}[section]
\newtheorem{theorem}{Theorem}[section]
\newtheorem{corollary}{Corollary}[section]
\newtheorem{proof}{Proof}
\newtheorem{remark}{Remark}[section]
\newtheorem{example}{Example}[section]
\newtheorem{proposition}{Proposition}[section]
\newtheorem{problem}{Problem}[section]

\makeatletter

\@addtoreset{equation}{section}
\makeatother

\begin{document}

\title{Galerkin Method with Trigonometric Basis on Stable Numerical Differentiation}
\author{Yidong Luo}
\address{School of Mathematics and Statistics, Wuhan University, Hubei Province, P. R. China}
\ead{Sylois@whu.edu.cn}
\vspace{10pt}
\begin{indented}
\item[]August 2017
\end{indented}

\begin{abstract}
This paper considers the $ p $ ($ p=1,2,3 $) order numerical differentiation on function $ y $ in $ (0,2\pi) $. They are transformed into corresponding Fredholm integral equation of the first kind. Computational schemes with analytic solution formulas are designed using Galerkin method on trigonometric basis. Convergence and divergence are all analysed in Corollaries 5.1, 5.2, and a-priori error estimate is uniformly obtained  in Theorem 6.1, 7.1, 7.2. Therefore, the algorithm achieves the optimal convergence rate $ O( \delta^{\frac{2\mu}{2\mu+1}} ) \ (\mu = \frac{1}{2} \ \textrm{or}  \ 1)$ with periodic Sobolev source condition of order $ 2\mu p $. Besides, we indicate a noise-independent a-priori parameter choice when the function $ y $ possesses the form of
\begin{equation*}
 \sum^{p-1}_{k=0} a_k t^k + \sum^{N_1}_{k=1} b_k \cos k t + \sum^{N_2}_{k=1} c_k \sin k t, \ b_{N_1}, c_{N_2} \neq 0,
\end{equation*}
In particular, in numerical differentiations for functions above, good filtering effect (error approaches 0) is displayed with corresponding parameter choice. In addition, several numerical examples are given to show that even derivatives with discontinuity can be recovered well.
\end{abstract}
%
%
%
%
%

\section{Introduction }
Numerical differentiation is a classical ill-posed problem which arises in different practical fields, such as option pricing, thermodynamics and  photoelectric response (See e.g. [4,11,13-16,25]). In process of numerical differentiation on a given function $ y(x) $ of specific smoothness,
always there would interfuse with a noise $ \delta y $ in measurement or calculations.
For this sake, it is routine to do numerical differentiation on the noisy function $ y^\delta := y +\delta y $, where the high frequency part in $ \delta y $ would bring uncontrolled huge error when computing with traditional numerical algorithms. In order to overcome the difficulties of ill-posedness, several kinds of regularization method were introduced.
\newline \indent Tikhonov method (See [8,11,12,17,26-28]) is a classical regularization method for numerical differentiation of first order. It is generally modeled that, for $ x \in [0,1] $ with a grid $ \Delta = \{ 0= x_0 < x_1 < \cdots < x_n =1 \} $ and $ h = \max_{i} x_{i+1} -x_i $ being its mesh size, given finite noisy samples $ \tilde y_i $ of $ y(x_i) $ such that $ \vert \tilde y_i - y(x_i) \vert \leq \delta $. Assume that $ \tilde y_0, \ \tilde y_n $ are exactly known boundary data, that is, $ \tilde y_0 = y(0), \ \tilde y_n = y(1) $. Then minimizing the cost functional
\begin{equation*}
\Psi[f]:= \frac{1}{n-1} \sum^{n-1}_{i=1} (\tilde y_i - f(x_i))^2 + \alpha \Vert f'' \Vert^2
\end{equation*}
in $ \{ f \in H^2(0,1): f(0) = y(0), f(1) = y(0) \} $ gives minimizer $ f_\alpha $. Afterward differentiating this minimizer gives $ f'_\alpha $ as the regularized solution to the exact derivative $ y' $ with appropriate parameter choice $ \alpha = \alpha(\delta) $. Further results illustrate that $ f'_{\alpha(\delta)} $ can converge to $ y' $ with best rate $ O(\sqrt{\delta}) $ with parameter choice for $ \alpha = \delta^2 $ (See [28]). However, we note that the penalty term $ \Vert f'' \Vert $ in cost functional basically demand that the all candidate solutions $ f'_\alpha $ must be at least $ H^1 $ smooth and further result in [27,28] illustrates that, for $ y \in C[0,1] / H^2(0,1) $, under specific parameter choice $ \alpha = \delta^2 $, the upper bound for $ \Vert f'_{\delta^2} - y' \Vert $ must tend to infinity as $ \delta, h \to 0 $. Thus this algorithm naturally deny to recover derivative with regularity less than $ H^1 $, especially  discontinuous derivative.
 \newline \indent Difference method [4,23] is another classical regularization method for numerical differentiation (including higher orders). It constructs difference quotient of $ p $ order as regularized solution to exact derivatives $ y^{(p)} $ with the stepsize $ h $ being the regularization parameter. The convergence of this scheme is established in $ L^\infty $ setting and will basically demand that $ y' \in C^{0,\alpha}, \alpha > 0$ (See [4]) which also deny to recover derivatives that are only continuous and discontinuous. Furthermore, the best convergence rate $ O(\delta^{\frac{2}{3}}) $ and $ O(\delta^{\frac{1}{3}}) $ for first and second order numerical differentiation are derived with $ h = O(\delta^{\frac{1}{3}}) $ respectively. But we need note the essential flaw in this algorithm that the numerical derivatives constructed by this algorithm will lose its smoothness and all be piecewise constant, whether the original function is smooth or not.
\newline \indent In this paper, we first formulate the $ p $ order derivative $ y^{(p)} $ as the unique solution of Fredholm integral equation of the first kind
 \begin{equation}
 A^{(p)} \varphi := \frac{1}{(p-1)!}\int^x_0 (x-t)^{p-1} \varphi (t) dt = y(x), \  x \in (0,2\pi).
 \end{equation}
  where $ A^{(p)}: L^2(0,2\pi) \to L^2(0,2\pi) $. For the simple design of computational scheme, we apply Galerkin method with trigonometric basis to above equation to construct a regularized solutions to $ y^{(p)} $ (refer to [9,19-22] for similar techniques).  The basic setting can be described as below:
  \newline \indent Assume that
 \begin{equation*}
 y^\delta, \  y , \ \delta y \in  L^2 (0,2\pi) \  \textrm{and} \ \ \Vert \delta y \Vert_{L^2} \leq \delta,
  \end{equation*}
  where $ \delta $ is noise level. Given a projection sequence $ \{  P_n \} $ which project $ L^2(0, 2\pi) $ onto
 subspace
 \begin{equation*}
X_n := span \{ \frac{1}{\sqrt{2\pi}}, \frac{ \cos t}{\sqrt{\pi}}, \frac{ \sin t}{\sqrt{\pi}},
\cdots, \frac{ \cos n t}{\sqrt{\pi}}, \frac{ \sin n t}{\sqrt{\pi}} \},
 \end{equation*}
we discrete (1.1) into a finite-rank approximation system
 \begin{equation}
 A^{(p)}_n \varphi_n = y_n  \  \  \varphi_n \in X_n, y_n := P_n y \in X_n,
 \end{equation}
 where $  A^{(p)}_n:= P_n A^{(p)}P_n: X_n \longrightarrow X_n $.
\newline \indent
Now solving (1.2) in sense of Moore-Penrose inverse gives $ {A_n^{(p)}}^\dagger P_n y $. This is a natural approximation scheme for $ y^{(p)} $, where $ \dagger $ denotes the Moore-Penrose inverse of linear operator. Considering the noise $ \delta y $, above scheme should be adjusted to a regularized scheme
\begin{equation}
 {A_n^{(p)}}^\dagger P_n y^\delta  \ \textrm{with} \ n^{(p)} = n^{(p)} (\delta),
 \end{equation}
where $ n^{(p)}:= n^{(p)} (\delta) $ is the regularization parameter choice such that
\begin{equation*}
n^{(p)}:= n^{(p)} (\delta) \to +\infty, \delta \to 0^+
\end{equation*}
and
  \begin{equation*}
  \Vert {A^{(p)}_{n^{(p)}(\delta)}}^\dagger P_{n^{(p)}(\delta)} y^\delta - y^{(p)} \Vert_{L^2} \to 0,  \ \delta \to 0^+.
  \end{equation*}
  Here notice that $ n^{(p)}:= n^{(p)} (\delta) \ (p=1,2,3) $ stands for parameter choice strategy of the first three order numerical
  differentiation respectively. Throughout this paper, without special indication, we follow this notation and $ p \in \overline{1,2,3}$.
 \newline \indent Main results of this paper and corresponding remarks are listed as follows.
\begin{itemize}
\item For $ y \in {\mathcal{H}}^p_0(0,2\pi) $ (this restriction on initial value data is removable), where $ {\mathcal{H}}^p_0(0,2\pi):= \{ y \in  H^p(0,2\pi): y(0) = \cdots = y^{(p-1)}(0) = 0 \}. $
a priori error estimate is obtained uniformly for first three order numerical differentiation as
\begin{equation*}
\Vert {A^{(p)}_n}^\dagger P_n y^\delta - y^{(p)}\Vert_{L^2(0,2\pi)}
\leq C^{(p)} n^p \delta + (\gamma^{(p)}+1)\Vert (I-P_n)y^{(p)} \Vert_{L^2},
\end{equation*}
This determines the parameter choice strategy:
\begin{equation*}
n^{(p)} = n^{(p)}(\delta) = \kappa \delta^{a - \frac{1}{p}},
\end{equation*}
where $ a \in (0,\frac{1}{p}) $ is optional and $ \kappa $ is a constant which depends on the concrete form of
$ \Vert (I - P_n) y^{(p)} \Vert_{L^2} $. This establish a convergence result for numerical differentiation of first three order when $ y^{(p)} \in L^2 $, especially for derivative with discontinuities. However, we need specify that, when recovering $ y^{(p)} \in L^2 $ are only continuous and discontinuous with no periodic smoothness, the constant $ \kappa $ is unknown and need to test in experiments (See section 8.3). In addition we give a notice that, whether the derivative is smooth or not, its approximation by above algorithm will be real analytic since it is a trigonometric polynomial.
\end{itemize}
\begin{itemize}
\item Supplemented with a priori information
\begin{equation*}
 y^{(p)} \in H_{per}^l(0,2\pi), \  l>0 \  \  (\textrm{periodic smoothness})
 \end{equation*}
above error estimate is strengthened into a more explicit form as
\begin{equation*}
\Vert {A^{(p)}_n}^\dagger P_n y^\delta - y^{(p)}\Vert_{L^2(0,2\pi)}
\leq C^{(p)} n^p \delta + (\gamma^{(p)}+1)\frac{1}{n^{l}} \Vert y^{(p)} \Vert_{H_{per}^{l}},
\end{equation*}
where $ C^{(p)} $, $ \gamma^{(p)} $ are all independent constants given in proceeding sections.
 And optimal convergence rate $ O(\delta^{\frac{2\mu}{2\mu+1}}) $ can be derived under periodic Sobolev source condition
\begin{equation*}
 y^{(p)} \in H_{per}^{2 \mu p } (0,2\pi)\ \  \mu = \frac{1}{2} \ \textrm{or}  \ 1,
 \end{equation*}
 with parameter choice $ n^{(p)} = \lambda^{(p)} \delta^{- \frac{1}{(2\mu + 1)p}} $, where $ \lambda^{(p)} $ is a constant only depends on exact derivative $ y^{(p)}$ and can be given explicitly in preceding sections 6,7. In particular, when
  \begin{equation*}
  y^{(p)} = a_0 + \sum^{N_1}_{k=1} b_k \cos kx + \sum^{N_2}_{k=1} c_k \sin kx \in H^\infty_{per}(0,2\pi)
  \end{equation*}
  the optimal parameter choice will degenerate to a constant $ n = \max(N_1,N_2) $ which does not depend on noise. Furthermore, the numerical study in section 8.1 demonstrates the good filtering effect (error approaches 0) occurs in this specific case.
\item In a more general setting for $ p $ order numerical differentiation when $ y \in H^p(0,2\pi) $, it is indicated in Corollary 6.2 that, when $ y \in H^p (0,2\pi) \setminus {\mathcal{H}}^p_0(0,2\pi) $,
 \begin{equation*}
  \Vert {A^{(p)}_n}^\dagger P_n y\Vert_{L^2} \longrightarrow \infty \ (n \to \infty).
  \end{equation*}
Now any parameter choice $ n^{(p)} (\delta) $ such that
$  n^{(p)} = n^{(p)} (\delta) \to \infty \ (\delta \to 0^+) $ may not be a proper regularization parameter since we can not determine
  \begin{equation*}
  \Vert {A^{(p)}_{n^{(p)}(\delta)}}^\dagger P_{n^{(p)}(\delta)} y^\delta - y^{(p)} \Vert_{L^2} \to 0,  \ \delta \to 0^+.
  \end{equation*}
through traditional estimate any more (See second point behind Corollary 5.2). In order to recover the regularization effect of algorithm, we introduce Taylor polynomial truncation of $ p -1 $ order to reform the regularized scheme, that is, using
\begin{equation*}
\bar y = y(x) - \sum^{p-1}_{k=0} \frac{y^{(k)}(0)}{k!} x^k \in {\mathcal{H}}^p_0(0,2\pi),
\end{equation*}
to replace $ y \in H^p(0,2\pi)$. In this way, the regularization effect can be well recovered (See section 6,7)with exact measurements on initial value data. Furthermore, we take possible noise in measurements in initial value data into consideration, and this effectively relax the requirement on precision of initial value data.
\end{itemize}
 \textbf{\textrm{Outline of Paper}}:
 In section 2, we introduce some tools and basic lemmas.
  In section 3, we illustrate general framework, give the main idea on how to utilize the noisy data $ y^\delta $ to
  recover the $ p $ order derivatives $ y^{(p)} $. In section 4, we give corresponding analytic solution formula to Galerkin approximation system which determines the well-posedness result and upper bound for noise error.
 In section 5, we propose an estimate on approximation error when RHS $ y $ belongs to $ \mathcal{H}^p_0 (0,2\pi) $, and give the convergence and divergence results with respect to $ y \in \mathcal{H}^p_0 (0,2\pi) $ and $ y \in L^2(0,2 \pi) \setminus \mathcal{H}^p_0 (0,2\pi) $ respectively.
 In sections 6 and 7, with periodic Sobolev source condition of order $ 2 \mu p $, we construct a priori error estimate and indicate the parameter choice strategy for optimal convergence rate $ O(\delta^{\frac{2\mu}{2\mu+1}})$ when  $ y \in \mathcal{H}^p_0 (0,2\pi) $ and $ y \in H^p(0,2\pi) \setminus \mathcal{H}^p_0 (0,2\pi) $  respectively. In section 8, we test some numerical examples to show the characteristics and effects of algorithm when derivatives are smooth and discontinuous respectively. In section 9, we conclude the main work of this paper.
\section{Preliminary and Basic Lemmas}
\subsection{ Moore-Penrose inverse}
Let $ X,Y $ be Hilbert space, and  $ A $ be bounded linear operator mapping from $ X $ to $ Y $. $ \mathcal{D}(A) $, $ \mathcal{N}(A) $ and $ \mathcal{R}(A) $ denote its domain, null space and range, respectively.
 \newline \indent For $ A: X \to Y $ and $ y \in \mathcal{R}(A) \oplus \mathcal{R}(A)^\perp $, the Moore-Penrose inverse $ x^\dagger:= A^\dagger y $ is defined as the element of smallest norm satisfying
\begin{equation*}
\Vert A x^\dagger -y \Vert = \inf \{  \Vert Ax - y \Vert | x \in X \}.
\end{equation*}
Thus $ A^\dagger: \mathcal{D}(A^\dagger) := \mathcal{R}(A) \oplus \mathcal{R}(A)^\perp \subseteq Y \longrightarrow  X $ defines a closed linear operator from $ Y $ to $ X $.
\newline \indent
In the following, we indicate some useful properties of Moore-Penrose inverse $ A^\dagger $:
 \begin{itemize}
 \item If $ A: X \to Y $ is one-to-one, then, for $ y \in \mathcal{R}(A) $, $ A^\dagger y $ naturally degenerates into $  A^{-1} y $.
 \item If $ \mathcal{R}(A) $ is closed, then $ \mathcal{D}(A^\dagger) = \mathcal{R}(A) \oplus \mathcal{R}(A)^\perp = Y $ and by closed graph theorem, $ A^\dagger: Y \to X $ is bounded.
\item  If $ \mathcal{R}(A) $ is closed, then $ A A^\dagger = P_{\mathcal{R}(A)}, \ A^\dagger A = P_{\mathcal{N}(A)^{\perp}}. $ If $ \mathcal{R}(A) $ is not necessarily closed, then the former identity need be adjusted into
     \begin{equation}
     A A^\dagger y = P_{\overline{\mathcal{R}(A)}} y, \ \forall y \in \mathcal{R}(A) \oplus \mathcal{R}(A)^\perp .
     \end{equation}
\end{itemize}
For more comprehensive information on Moore-Penrose inverses, see [2, Chapter 9] or [6,7].
\subsection{Sobolev spaces}
Throughout this paper, we only discuss on Sobolev space over $ \mathrm{R} $. Without specification, we denote $ H^p(0,2\pi):= H_{\mathrm{R}}^p(0,2\pi) $.  Here we introduce all kinds of notations of Sobolev spaces which will be used in the context. For more information, see [1,5] and [3, Appendix 4].
\subsubsection{Sobolev spaces of integer order}
  For some positive integer $ p $, the Sobolev space $ H^p(0,2\pi) $ is defined as
    \begin{equation}
   H^p (0,2\pi) := \{ y \in L^2 (0,2\pi): D^1 y, \cdots, D^p y \in L^2(0,2\pi)\},
   \end{equation}
  where $ D^k y $ means weak derivative, defined as $ \zeta \in L^2(0,2\pi) $ which satisfies
  \begin{equation*}
  \int^{2\pi}_0 \zeta \varphi dx = (-1)^k \int^{2\pi}_0 y \varphi^{(k)} dx, \ \varphi \in C^\infty_0(0,2\pi).
  \end{equation*}
  Equivalently, it can be characterized in absolute continuous form (refer to [3, Page 14]) as
  \begin{equation*}
    H^p(0,2\pi) = \mathcal{U}^p[0,2\pi] := \{ y \in C^{p-1} [0,2\pi]:
   \end{equation*}
   \begin{equation}
   \textrm{there exists} \ \Psi \in L^2(0,2\pi) \ \textrm{ such that} \  y^{(p-1)}(x) = \alpha + \int^x_0 \Psi(t) dt, \alpha \in \mathrm{R} \}
   \end{equation}
   Here notice that above $ " = "$ admits a possible change in a set of measure zero. In this paper, when it concerns Sobolev functions of one variable $ y \in H^p(0,2\pi) $, we, by default, modify $ y \in H^p(0,2\pi) \ (p \in \mathrm{N})$ in a set of measure zero such that it belongs to the latter fine function space $ \mathcal{U}^p[0,2\pi] $.
   \newline \indent Besides, for $ p \in \mathrm{N} $, we define
     \begin{equation*}
   {\mathcal{H}}^p_{z}(0,2\pi):= \{y \in H^p (0,2\pi): y(z) = \cdots = y^{(p-1)}(z) = 0 \}, \ z = 0 \ \textrm{or} \ 2\pi
   \end{equation*}
   and
   \begin{equation*}
   {\dot \mathcal{H}}^{2p}(0,2\pi):= \{y \in H^{2p} (0,2\pi):
   \end{equation*}
   \begin{equation*}
   y(2\pi) = \cdots = y^{(p-1)}(2\pi) = y^{(p)}(0) = \cdots = y^{(2p-1)}(0) = 0 \}.
   \end{equation*}
   \subsubsection{Fractional periodic Sobolev spaces }
 \indent For real number $ s > 0 $, periodic Sobolev spaces of fractional order $  H_{per}^s(0,2\pi) $ is defined in trigonometric form as
 \begin{equation*}
 H_{per}^s(0,2\pi) :=
 \{ \varphi \in L^2(0,2\pi):  \xi^2_0 +\sum^\infty_{k =1} (1+k^2)^s (\xi^2_k + \eta^2_k) < \infty \},
 \end{equation*}
 where
 \begin{equation*}
 \xi_0 = \frac{1}{\sqrt{2\pi}} \int^{2\pi}_0 \varphi(t)dt , \
  \xi_k = \frac{1}{\sqrt{\pi}} \int^{2\pi}_0 \varphi(t) \cos k t dt, \
  \eta_k = \frac{ 1}{\sqrt{\pi}} \int^{2\pi}_0 \varphi(t) \sin k t dt.
 \end{equation*}
 Supplementing another element $ \psi \in  H_{per}^s (0,2\pi) $, its inner product is rephrased as
\begin{equation*}
(\varphi,\psi)_{H_{per}^s} =  \xi_0 \zeta_0 +\sum^\infty_{k =1} (1+k^2)^s (\xi_k \zeta_k + \eta_k \lambda_k)
\end{equation*}
with
\begin{equation*}
 \zeta_0 = \frac{1}{\sqrt{2\pi}} \int^{2\pi}_0 \psi(t)dt , \
  \zeta_k = \frac{1}{\sqrt{\pi}} \int^{2\pi}_0 \psi(t) \cos k t dt, \
  \lambda_k = \frac{ 1}{\sqrt{\pi}} \int^{2\pi}_0 \psi(t) \sin k t dt.
  \end{equation*}
  In addition, we define
  \begin{equation*}
   H^{\infty}_{per} (0,2\pi) := \bigcap_{s>0} H^s_{per} (0,2\pi).
   \end{equation*}
\subsection{ Integro-differential operator of $ p $ order }
Define integro-differential operator of integer order $ p $ as:
\begin{equation*}
A^{(p)}: L^2(0, 2\pi) \longrightarrow L^2(0,2\pi)
\end{equation*}
\begin{equation}
 \varphi \longmapsto ( A^{(p)}\varphi)(x):= \frac{1}{\Gamma(p)}\int^x_0 (x-t)^{p-1} \varphi (t) dt, \  x \in (0,2\pi).
  \end{equation}
This is a compact linear operator with infinite-dimensional range, which satisfies
\begin{lemma}
\indent
\begin{equation*}
{\mathcal{H}}^p_0(0,2\pi) = \mathcal{R}(A^{(p)}) , \ p=1,2,3.
\end{equation*}
\end{lemma}
\begin{proof}
"$\subseteq$": Assume that $ y \in  {\mathcal{H}}^p_0(0,2\pi) \ (p=1,2,3) $. There exists a  $ y^\star \in \mathcal{U}^p[0,2\pi] $ as a modification of $ y $ in a possible $ 0 $ measure set with $ 0 $ initial value data, that is, $ y^\star(0) = \cdots = {y^\star}^{(p-1)}(0) = 0$.  With integration formula by parts, it is not difficult to verify that,
 \begin{equation*}
 \frac{1}{(p-1)!}  \int^x_0  ( x - t )^{p-1} {y^\star}^{(p)} (t) dt = y^\star(x) = y(x), \ a.e. .
 \end{equation*}
Thus,
 \begin{equation*}
{\mathcal{H}}^p_0(0,2\pi)\subseteq \mathcal{R}(A^{(p)}), \ p =1, 2, 3 .
 \end{equation*}
 \indent "$\supseteq$": For simplicity, we only provide proof of case $ p = 2 $.
 Assuming $ y \in \mathcal{R}(A^{(2)}) $, then there exists a $ \varphi \in L^2(0,2\pi) $ such that
\begin{equation*}
  \int^x_0  ( x - t ) \varphi (t) dt = y(x), \ a.e. .
 \end{equation*}
It is not difficult to verify that
 \begin{equation*}
 D^1 y =  \int^x_0 \varphi(t) dt, \ D^2 y = \varphi(t) \  a.e.
 \end{equation*}
 With definition of (2.2), it yields that $ y \in H^2 (0,2\pi) $. Then by absolute continuous characterization (2.3) of Sobolev function of one variable, there exist a $ y^\star \in \mathcal{U}^2 [0,2\pi] $ as modification of $ y $ in a $ 0 $ measure set. Thus we have
 \begin{equation*}
 y^\star =   \int^x_0  ( x - t ) \varphi (t) dt, \  {y^\star}' = D^1 y^\star = D^1 y = \int^x_0 \varphi(t) dt, \ a.e.
  \end{equation*}
 Notice that
 \begin{equation*}
  y^\star, \  {y^\star}', \  \int^x_0  ( x - t ) \varphi (t) dt, \ \int^x_0 \varphi(t) dt
  \end{equation*}
  are all continuous functions, thus
 \begin{equation*}
 y^\star =  \int^x_0  ( x - t ) \varphi (t) dt, \  {y^\star}' = \int^x_0 \varphi(t) dt. \ (\textrm{strictly})
 \end{equation*}
\indent  "$ \supseteq $" holds for case $ p =2 $.
\end{proof}
\indent With above equality, we describe the density of range in $ L^2(0, 2\pi) $.
\begin{lemma}
$  $\begin{equation*} \overline{\mathcal{R}(A^{(p)})} = L^2(0,2\pi), \
P_{\overline{\mathcal{R}(A^{(p)})}} =I \  \  ( p=1,2,3 ),
\end{equation*}
where $ I $ is the identity operator on $ L^2 (0 , 2\pi) $.
\end{lemma}
\begin{proof}
With Lemma 2.1, $ {\mathcal{H}}^p_0(0,2\pi) = \mathcal{R}(A^{(p)}). $
Recall the fact that $ C^\infty_0(0,2\pi) $ is dense in $ L^2(0,2\pi) $
and notice that $ C^\infty_0 (0,2\pi) \subseteq {\mathcal{H}}^p_0(0,2\pi) $, then $ L^2(0,2\pi) = \overline{C^\infty_0 (0,2\pi) } \subseteq \overline{{\mathcal{H}}^p_0(0,2\pi)} \subseteq L^2(0,2\pi). $
The result follows.
\end{proof}
This implies that $ \mathcal{R}(A^{(p)})^\perp = {0} $, and
 \begin{equation*}
 \mathcal{D}({A^{(p)}}^\dagger)=  \mathcal{R}(A^{(p)}) \oplus \mathcal{R}(A^{(p)})^\perp =  \mathcal{R}(A^{(p)}) ={\mathcal{H}}^p_0(0,2\pi).
\end{equation*}
Now differentiating the both sides of equation (1.1) for $ y \in \mathcal{H}^p_0(0,2\pi) $ in $ p $ order yields that
$ {A^{(p)}}^\dagger y =  {A^{(p)}}^{-1} y = y^{(p)} $. This gives
\begin{lemma}
$ {A^{(p)}}^\dagger y = y^{(p)}, \ \forall y \in {\mathcal{H}}^p_0(0,2\pi).$
\end{lemma}
 \subsection{ Galerkin Projection scheme with Moore-Penrose inverses }
\indent Let $ X $ be Hilbert space. For the linear operator equation
 \begin{equation*}
 A\varphi=y,
 \end{equation*}
 where $  A : X \longrightarrow X $ is bounded and linear. To approximate
  \begin{equation*}
   \varphi^\dagger := A^\dagger y \in X,
   \end{equation*}
 \indent We introduce a sequence of finite-dimensional subspaces $ \{ X_n \}$, which satisfies
 \begin{equation*}
  X_n \subseteq X_{n+1} , \  \overline{\bigcup^\infty_{n=1} X_n} =X.
 \end{equation*}
 Then construct a sequence of orthogonal projections $ \{ P_n \} $,
 where $ P_n $ projects $ X $ onto $ X_n $, and gives Galerkin approximation setting
\begin{equation}
 A_n \varphi_n = y_n, \ \ y_n := P_n y  \in X_n,
\end{equation}
where $ A_n := P_n A P_n : X_n \longrightarrow X_n $.
 Hence solving (2.5) in sense of Moore-Penrose inverse gives Galerkin projection scheme
\begin{equation}
 \varphi^\dagger_n := A^\dagger_n y_n \in X_n,
  \end{equation}
  where $ A^\dagger_n: \mathcal{R}(A_n) + {\mathcal{R}(A_n)}^{\perp_n} = X_n \longrightarrow  X_n $.
  Notice that $ \perp_n $ means orthogonal complement in finite dimensional Hilbert space $ X_n $.
\newline \indent Now $ \{ \varphi^\dagger_n\} $ is a natural approximate scheme for $ \varphi^\dagger $.
To study its convergence property, we introduce the Groetsch regularizer for setting (2.5) as
\begin{equation*}
R_n := A_n^\dagger P_n A : X \longrightarrow X_n \subseteq X,
\end{equation*}
define the Groetsch regularity as $ \ \sup\limits_n \Vert R_n \Vert < +\infty $, and introduce the following result:
\begin{lemma}
For above Galerkin approximate setting (2.5), if Groetsch regularity holds, then
\newline (a) For $ y \in \mathcal{D} ( A^\dagger ) = \mathcal{R}(A) + \mathcal{R}(A)^\perp $.
\begin{equation}
  \Vert A^\dagger_n P_n P_{\overline{\mathcal{R}(A)}} y - A^\dagger y \Vert \leq \Vert P_{\mathcal{N} ( A_n )} A^\dagger y \Vert + \Vert R_n -I_X \Vert \Vert (I-P_n) A^\dagger y \Vert,
\end{equation}
(b) For  $ y \notin \mathcal{D}(A^\dagger) $,
 \begin{equation*}
  \lim_{n \to \infty} \Vert A^\dagger_n P_n P_{\overline{\mathcal{R}(A)}} y \Vert = \infty.
  \end{equation*}
\end{lemma}
\begin{proof}
see[18, theorem 2.2]
\end{proof}
\subsection{Higher order estimate under trigonometric basis}
 \indent To further estimate the right term $ \Vert (I-P_n) A^\dagger y \Vert $ under trigonometric basis in $ L^2 $, similar to the result [3, Lemma A.43],
we introduce another error estimate:
\begin{lemma}
Let \begin{math} P_n : L^2 ( 0, 2\pi ) \longrightarrow X_n \subset L^2(0, 2\pi ) \end{math} be an orthogonal projection operator, where
\begin{equation*}
X_n :=
\{
\xi_0 \cdot \frac{1}{\sqrt{2\pi}} + \sum\limits^{n}_{ k = 1 }\xi_k  \frac{\cos  k t}{\sqrt{\pi}}
+ \sum\limits^n_{ k=1 } \eta_k  \frac{\sin  k t}{\sqrt{\pi}} : \ \xi_0, \xi_k , \eta_k \in \mathrm{R}
\}.
\end{equation*}
 Then \begin{math} P_n \end{math} is given as follows£º
\begin{equation*}
(P_n x)(t)
= \xi_0 \frac{1}{\sqrt{2\pi}} + \sum\limits^{n}_{ k = 1 }\xi_k \frac{\cos k t}{\sqrt{\pi}}
+ \sum^n_{ k=1 } \eta_k \frac{\sin k t}{\sqrt{\pi}} ,
\end{equation*}
where
\begin{equation*}
 \xi_0 = \frac{1}{\sqrt{2\pi}} \int^{2\pi}_0 x(t) dt ,\
  \xi_k = \int^{2\pi}_0 x(t) \frac{ \cos k t}{\sqrt{\pi}} dt,
  \end{equation*}
  \begin{equation*}
  \eta_k = \int^{2\pi}_0 x(t) \frac{ \sin k t}{\sqrt{\pi}} dt, \ \ 1 \leq k \leq n
  \end{equation*}
 are the Fourier coefficients of \begin{math} x \end{math}. Furthermore, the following estimate holds:
 \begin{equation*}
 \Vert x- P_n x \Vert_{L^2} \leq \frac{1}{ n^r } \Vert x \Vert_{H_{per}^r} \ \ for \ all \  x \in \  H_{per}^r(0,2\pi),
 \end{equation*}
 where \begin{math} r \geq 0 \end{math}.
\end{lemma}
\section{General Framework }
We start from
\begin{problem}
 Assume that we have $ y \in {\mathcal{H}}^p_0(0,2\pi) $ and $ y^\delta $ measured on $(0, 2\pi) $,
belonging to $ L^2(0,2\pi) $ such that $ \Vert y^\delta - y \Vert_{L^2} \leq \delta $.
How to get a stable approximation to $ y^{(p)} $  ?
 \end{problem}
 In Lemma 2.3, we have known that $ y^{(p)} $ is the solution of linear operator equation
\begin{equation}
 A^{(p)} \varphi := \frac{1}{\Gamma(p)}\int^x_0 (x-t)^{p-1} \varphi (t) dt = y(x), \  x \in (0,2\pi),
\end{equation}
when
\begin{equation*}
y \in {\mathcal{H}}^p_0(0,2\pi) := \{ y \in H^p(0,2\pi) : y(0)= \cdots = y^{(p-1)}(0) =0\}.
\end{equation*}
\subsection{Formulation of finite-dimensional approximation system}
In the following, we consider to approximate $ y^{(p)} $ by the Galerkin method. Set
\begin{equation*}
A^{(p)} : X=L^2(0, 2\pi) \longrightarrow  L^2 (0, 2\pi).
\end{equation*}
\indent Choose a sequence of orthogonal projection operators $ \{ P_n \} $, where
$ P_n $ projects $ L^2(0, 2\pi ) $ onto
\begin{equation*}
X_n := span \{ \frac{1}{\sqrt{2\pi}}, \frac{ \cos t}{\sqrt{\pi}}, \frac{ \sin t}{\sqrt{\pi}},
\cdots, \frac{ \cos n t}{\sqrt{\pi}}, \frac{ \sin n t}{\sqrt{\pi}} \}.
 \end{equation*}
 Then degenerate the original operator equation with noisy data
\begin{equation*}
A^{(p)} \varphi =y^\delta
\end{equation*}
into a finite-rank system
\begin{equation}
 A_n^{(p)} \varphi_n =y^\delta_n,
\end{equation}
where
 \begin{equation}
 A_n^{(p)}:=P_n A^{(p)} P_n : X_n \longrightarrow X_n  , \ \
    y^\delta   _n := P_n y^\delta.
 \end{equation}
Span $ A_n^{(p)} $ under above basis, then the finite-rank system (3.2) is transformed into the linear system as
\begin{equation}
M_n^{(p)} u_n = b^\delta_n, \ u_n, b^\delta_n \in \mathrm{R}^{2n+1} .
\end{equation}
Notice that $ M_n^{(p)} $ and $ b^\delta_n $ are defined as follows:
\begin{equation*}
M_n^{(p)} := (m_{ij}^{(p)})_{(2n+1) \times (2n+1)}
\end{equation*}
where
  \begin{equation*}
m_{ij}^{(p)} := (A^{(p)}_n(\xi_j), \xi_i)_{L^2}, \ i,j \in \overline{0,1,2,\cdots,2n-1,2n}
\end{equation*}
\begin{equation*}
\xi_0 := \frac{1}{\sqrt{2\pi}}, \xi_{2k-1} := \frac{\cos k x}{\sqrt{\pi}}, \xi_{2k}:= \frac{\sin k x}{\sqrt{\pi}}, \ k \in \overline{1,2,\cdots,n}.
\end{equation*}
Indeed,
\begin{equation*}
A_n^{(p)} (\frac{1}{\sqrt{2\pi}}, \frac{\cos t }{\sqrt{\pi}}, \frac{\sin t }{\sqrt{\pi}}, \cdots , \frac{\cos n t }{\sqrt{\pi}}, \frac{\sin n t }{\sqrt{\pi}} )
\end{equation*}
\begin{equation}
= (\frac{1}{\sqrt{2\pi}}, \frac{\cos t }{\sqrt{\pi}}, \frac{\sin t }{\sqrt{\pi}}, \cdots , \frac{\cos n t }{\sqrt{\pi}}, \frac{\sin n t }{\sqrt{\pi}} ) M_n^{(p)}.
\end{equation}
And $ b^\delta_n:= (f_0,f_1,g_1,\cdots, f_n, g_n)^T $ is defined as
\begin{equation*}
f_0 := \int^{2\pi}_0 y^\delta (t) \frac{1}{\sqrt{2\pi}} dt
\end{equation*}
\begin{equation*}
 f_k := \int^{2\pi}_0 y^\delta (t) \frac{\cos kt }{\sqrt{\pi}} dt, \
g_k := \int^{2\pi}_0 y^\delta (t) \frac{\sin kt }{\sqrt{\pi}} dt, k \in \overline{1,2,\cdots,n}.
\end{equation*}
Indeed,
\begin{equation*}
y^\delta_n= (\frac{1}{\sqrt{2\pi}}, \frac{\cos t }{\sqrt{\pi}}, \frac{\sin t }{\sqrt{\pi}}, \cdots , \frac{\cos n t }{\sqrt{\pi}}, \frac{\sin n t   }{\sqrt{\pi}} ) b_n^\delta.
\end{equation*}
Once we figure out \begin{math}{ u_n^{p, \delta}} = {M_n^{(p)}}^\dagger b^\delta_n \end{math},
then we obtain solution for (3.2),
\begin{equation*}
{ \varphi_n^{p, \delta}} = (\frac{1}{\sqrt{2\pi}}, \frac{ \cos t}{\sqrt{\pi}}, \frac{ \sin t}{\sqrt{\pi}}, \cdots, \frac{ \cos n t}{\sqrt{\pi}}, \frac{ \sin n t}{\sqrt{\pi}}) {u^{p,\delta}_n} .
\end{equation*}
in sense of Moore-Penrose inverse, \begin{math} {\varphi_n^{p, \delta}} = {A_n^{(p)}}^\dagger y^\delta_n \end{math}.
This is the regularized scheme. In the following, we need to determine a regularization parameter $ n^{(p)} = n^{(p)}(\delta) $ such that
 \begin{equation*}
  {\varphi_{n^{(p)}(\delta)}^{p, \delta}} := {A_{n^{(p)}(\delta)}^{(p)}}^\dagger y^\delta_{n^{(p)}(\delta)} \stackrel{s}{\longrightarrow}
  y^{(p)}, \ \delta \to 0^+.
  \end{equation*}
\subsection{Total error estimate and parameter choice for regularization}
\indent Now, in order to control the accuracy of computation, we adjust parameter choice
strategy $ n^{(p)} = n^{(p)}(\delta) $ according to following total error estimate
\begin{equation}
\Vert \varphi^{p, \delta}_n - y^{(p)} \Vert_{L^2} := \Vert {A^{(p)}_n}^\dagger P_n y^\delta -  y^{(p)} \Vert_{L^2}.
\end{equation}
Since Lemma 2.3 illustrates that
\begin{equation}
  y^{(p)} = { A^{(p)} }^\dagger y,  \  y \in {\mathcal{H}}^p_0(0,2\pi),
\end{equation}
inserting (3.7) into (3.6), the formula (3.6) becomes
\begin{equation*}
\Vert {A^{(p)}_n}^\dagger P_n  y^\delta - {A^{(p)}}^\dagger y \Vert_{L^2}, \ y \in {\mathcal{H}}^p_0(0,2\pi).
\end{equation*}
Throughout this paper we use the following definitions
\begin{itemize}
\item Total error
\begin{equation*}
e^{(p)}_T:=
\Vert {A^{(p)}_n}^\dagger P_n  y^\delta - {A^{(p)}}^\dagger y \Vert_{L^2},
\end{equation*}
\end{itemize}
which is broken into two parts (c.f.[10, Chapter 1.1]):
 \begin{itemize}
 \item  Noise error:
\begin{equation*}
e^{(p)}_N:=
\Vert {A^{(p)}_n}^\dagger P_n  y^\delta - {A^{(p)}_n}^\dagger P_n  y  \Vert_{L^2}
\end{equation*}
\item Approximation error:
\begin{equation*}
e^{(p)}_A:=
\Vert {A^{(p)}_n}^\dagger P_n  y - {A^{(p)}}^\dagger y \Vert_{L^2}
\end{equation*}
\end{itemize}
It is an easy observation that $ e^{(p)}_T \leq e^{(p)}_N + e^{(p)}_A $. Upon this fact, we figure
out the total error estimate by estimating $e^{(p)}_N $ and $ e^{(p)}_A $ respectively.
\section{Well-posedness and numerical scheme of Galerkin System}
With concrete expressions of $ M^{(p)}_n $ in Appendix A, it is not difficult to obtain:
 \begin{theorem}
 Finite dimensional system (3.4) is well-posed, that is, there exists a unique solution to (3.4),
denoted as
\begin{equation*}
u^{p,\delta}_n = {M^{(p)}_n}^{-1} b^\delta_n,
\end{equation*}
where
\begin{equation*}
b^\delta_n =(f_0, f_1, g_1, \cdots, f_n, g_n)^T,
u^{p,\delta}_n =(\xi^{(p)}_0, \xi^{(p)}_1, \eta^{(p)}_1, \cdots, \xi^{(p)}_n, \eta^{(p)}_n )^T.
\end{equation*}
Moreover, analytic formulas for the solution of  Galerkin approximation system (3.2) are determined as follows:
\begin{equation*}
{A^{(p)}_n}^\dagger P_n y = \xi^{(p)}_0 \frac{1}{\sqrt{2\pi}} + \sum^{n}_{k=1} \xi^{(p)}_k  \frac{\cos k t}{ \sqrt{\pi} } + \sum^n_{k=1}\eta^{(p)}_k \frac{ \sin k t}{\sqrt{\pi}}.
\end{equation*}
Corresponding three cases are listed as follows.
\newline \textbf{Case} $ p=1 $:
\begin{equation}
\xi^{(1)}_0 = \frac{1}{\pi}( f_0 + \sqrt{2} \sum_{k=1}^n f_k  ),
\end{equation}
\begin{equation}
\xi^{(1)}_k =  \sqrt{2} \xi^{(1)}_0 +k g_k,
\end{equation}
\begin{equation}
\eta^{(1)}_k = - k f_k.
\end{equation}
\textbf{Case} $ p=2 $:
\begin{equation}
\xi^{(2)}_0 = \frac{L^{-1}_n}{4\pi^2} (f_0 + \sqrt{2} \sum_{k=1}^n f_k + \frac{\sqrt{2} \pi}{2n+1}\sum_{k=1}^n k g_k),
\end{equation}
\begin{equation}
\xi^{(2)}_k = \sqrt{2} \xi^{(2)}_0 -k^2 f_k,
\end{equation}
\begin{equation}
\eta^{(2)}_k =  \frac{2k}{2n+1} \sum^n_{k=1} kg_k -k^2 g_k - \frac{\sqrt{2} k \pi}{ 2n+1} \xi^{(2)}_0,
\end{equation}
where
\begin{equation*}
L_n:= \frac{1}{6}+  \frac{1}{2\pi^2} S_n -\frac{1}{4} \frac{2n}{2n+1}, \
S_n := \sum^n_{k=1} \frac{1}{k^2}.
\end{equation*}
\textbf{Case} $ p=3 $:
\begin{equation}
\xi^{(3)}_0 = \frac{T^{-1}_n}{4\pi^3} (f_0 + \sqrt{2} \sum\limits_{k=1}^n f_k + \frac{\sqrt{2} \pi}{2n+1}\sum\limits_{k=1}^n k g_k -F_n \sum\limits_{k=1}^n k^2 f_k),
\end{equation}
\begin{equation}
\xi^{(3)}_k =  - k^3 g_k +\frac{2}{2n+1} k^2 \sum\limits_{k=1}^n k g_k - \frac{2\pi k^2 }{(2n+1)^2} \sum^n_{k=1} k^2 f_k + \varepsilon_{n,k} \xi^{(3)}_0
\end{equation}
\begin{equation}
\eta^{(3)}_k= k^3 f_k  - \frac{ 2 k}{2n+1} \sum\limits_{k=1}^n k^2 f_k - \frac{\sqrt{2} \pi k}{2n+1} \xi^{(3)}_0,
\end{equation}
where
\begin{equation*}
T_n := \frac{1}{12}+ \frac{1}{2n+1} \frac{1}{\pi^2} S_n -\frac{1}{3} \frac{2n}{2n+1}+ \frac{n^2}{(2n+1)^2},
\end{equation*}
\begin{equation*}
F_n :=\frac{4\sqrt{2} \pi^2 }{2n+1} L_n, \ K_n := -2\pi^2 L_n, \varepsilon_{n,k} = \sqrt{2} (1+ \frac{2 k^2}{2n+1}K_n).
\end{equation*}
  \end{theorem}
\begin{remark}
After we solve the Galerkin approximation system, we know that $ A^{(p)}_n: X_n \to X_n $ is one-to-one and surjective. For the usage of the proceeding section, we claim that $ \mathcal{N}( A^{(p)}_n ) = 0, \ \mathcal{R}(A^{(p)}) = X_n.$
\end{remark}
\begin{remark}
\indent With above analytic formulas, it is not difficult to figure out that
 \begin{equation}
 \Vert {A_n^{(p)}}^\dagger\Vert_{X_n \to X_n } \leq C^{(p)} n^p
 \end{equation}
 where $ C^{(1)} = \sqrt{3},   \  C^{(2)} \approx 11.8040,  \  C^{(3)} \approx 345.0754 $.  Here we specify that $ \Vert \cdot \Vert_{X_n} $ is induced by $ \Vert \cdot \Vert_{L^2} $, that is,  $ \Vert x_n \Vert_{X_n} := \Vert x_n \Vert_{L^2}$, $ \forall x_n \in X_n $. This give bound to the estimate of noise error.
\end{remark}

\section{Estimate on Approximation Error and Instability result}
 We use Lemma 2.4 to analyse the convergence and divergence of Galerkin method. The key point is the estimate of
  \begin{equation}
 \sup_n \Vert R^{(p)}_n \Vert < +\infty, \textrm{where} \  R^{(p)}_n:= {A^{(p)}_n}^\dagger P_n A^{(p)}.
 \end{equation}
 To gain an uniform upper bound for above formula, we first prepare two decay estimate of
 \begin{equation*}
  R^{(p)}_n(\frac{ \cos j t}{\sqrt{ \pi }} ) \ \textrm{and} \  R^{(p)}_n(\frac{ \sin j t}{\sqrt{ \pi }} )
 \end{equation*}
 with respect to integer variable $ j $:
\begin{lemma}
 For operators $ A^{(p)}, A^{(p)}_n $ defined in (2.4),(3.3) respectively, set
 \begin{equation*}
 ({A^{(p)}_n}^\dagger P_n A^{(p)} ( \frac{ \cos j t}{\sqrt{ \pi }} )(t)
 = \alpha^{(p)}_0 \frac{1}{\sqrt{2\pi}} + \sum^n_{k=1} \alpha^{(p)}_k \frac{ \cos k t}{\sqrt{ \pi }}
 + \sum^n_{k=1} \beta^{(p)}_k \frac{ \sin k t}{\sqrt{ \pi }}.
 \end{equation*}
When $ j \geq n+1 $,
  \begin{equation*}
 \alpha^{(p)}_0
 = \alpha^{(p)}_0 (n, j) = (({A^{(p)}_n}^\dagger P_n A^{(p)} ( \frac{ \cos j t}{\sqrt{ \pi }} ), \frac{1}{\sqrt{2\pi}}  )_{L^2},
 \end{equation*}
  \begin{equation*}
 \alpha^{(p)}_k
 =  \alpha^{(p)}_k (n, j) = (({A^{(p)}_n}^\dagger P_n A^{(p)} ( \frac{ \cos j t}{\sqrt{ \pi }} ), \frac{ \cos k t}{\sqrt{ \pi }} )_{L^2},
 \end{equation*}
 \begin{equation*}
 \beta^{(p)}_k
 = \beta^{(p)}_k (n, j) = (({A^{(p)}_n}^\dagger P_n A^{(p)} ( \frac{ \cos j t}{\sqrt{ \pi }} ), \frac{ \sin k t}{\sqrt{ \pi }} )_{L^2},
 \end{equation*}
and
\begin{equation}
\vert \alpha^{(p)}_0 \vert \leq   \frac{C_1^{(p)}}{j} , \ \vert \alpha^{(p)}_k \vert \leq  \frac{C_2^{(p)}}{j} , \vert \beta^{(p)}_k \vert \leq \frac{C_3^{(p)}}{j}, \  1 \leq k \leq n .
\end{equation}
where
\begin{eqnarray*}
C_1^{(1)}=0 , \  C_1^{(2)}=\sqrt{2}, \  C_1^{(3)}= 11\sqrt{2}, \\
C_2^{(1)}= 0, \ C_2^{(2)}= 2, \  C_2^{(3)}=23, \\
C_3^{(1)}=0, \  C_3^{(2)}= \pi, \ C_3^{(3)}=11\pi.
\end{eqnarray*}
When $ p =3 $, we need an extra condition $ n \geq 5 $ to maintain above estimate.
\end{lemma}
\begin{proof}
 \textbf{Case $ p=1 $}: When \begin{math} j \geq n+1  \end{math}, substituting (B.1) into (4.1),(4.2) and (4.3),
 it follows that
\begin{equation*}
({A^{(1)}_n}^\dagger P_n A^{(1)} ( \frac{ \cos j t}{\sqrt{ \pi }} ))(t) = {A^{(1)}_n}^\dagger (0) =0.
\end{equation*}
This gives lemma for case $ p=1 $.
\newline \textbf{ Case $ p=2 $}:
Inserting (B.2) into (4.4), it follows that
\begin{equation*}
\alpha^{(2)}_0 = \frac{ \sqrt{2} L^{-1}_n }{  4 \pi^2 j^2  }.
  \end{equation*}
  Hence
\begin{equation}
 0 \leq  \alpha^{(2)}_0 \leq \frac{\sqrt{2}}{j}  \ ( \textrm{by (C.1)}).
\end{equation}
Besides, inserting (B.2) into (4.5),(4.6) respectively, it yields that
 \begin{equation*}
 \alpha^{(2)}_k= \sqrt{2}\alpha^{(2)}_0,
\beta^{(2)}_k = - \frac{\sqrt{2} k \pi }{ 2n+1 } \alpha^{(2)}_0.
\end{equation*}
Then, by (5.3),
 \begin{equation*} 0 \leq  \alpha^{(2)}_k \leq \frac{2}{j}, \
- \frac{\pi}{j} \leq  \beta^{(2)}_k \leq  0.
 \end{equation*}
\textbf{ Case $ p=3 $}:
Inserting (B.3) into (4.7), it follows that
\begin{equation*}
\alpha^{(3)}_0 = \frac{T^{-1}_n}{4\pi^3} \frac{\sqrt{2} \pi }{j^2} \frac{1}{2n+1}.
\end{equation*}
Notice Proposition C.1 (C.4),
\begin{equation*}
T_n \in [ \frac{1}{396} \frac{1}{n}\frac{1}{2n+1}, \frac{3}{40} \frac{1}{n(2n+1)}] , \ n \geq 5.
\end{equation*}
Hence,
\begin{equation}
 0  \leq \alpha^{(3)}_0 \leq \frac{11 \sqrt{2}}{j}, \  \textrm{where} \ \alpha^{(3)}_0 := \alpha^{(3)}_0 ( n, j ), \textrm{and} \ n \geq 5.
\end{equation}
 \indent Besides, insert (B.3) into (4.8), then it follows that
\begin{equation}
\alpha^{(3)}_k =      \frac{1}{2n+1} \frac{2k^2}{j^2} + \sqrt{2} (1+ \frac{2 k^2}{2n+1}K_n) \alpha^{(3)}_0.
\end{equation}
By Proposition C.1(C.2), it is routine to obtain that
\begin{equation*}
\frac{2 k^2}{2n+1}K_n \in  [-2,0] , \ 1+\frac{2 k^2}{2n+1}K_n \in [-1,1].
\end{equation*}
Hence, with (5.4),
\begin{equation*}
0 \leq  \vert \alpha^{(3)}_k \vert \leq  \frac{23}{j}.
\end{equation*}
 Further, insert (B.3) into (4.9), and we have
\begin{equation*}
\beta^{(3)}_k = - \frac{\sqrt{2} \pi k}{2n+1} \alpha^{(3)}_0.
\end{equation*}
Hence
\begin{equation*}
 - \frac{11\pi}{j} \leq \beta^{(3)}_k \leq  0 \ \  (\textrm{by (5.4)}).
\end{equation*}
\end{proof}
\begin{lemma}
 For operators $ A^{(p)}, A^{(p)}_n $ defined in (2.4),(3.3) respectively, set
 \begin{equation*}
 ({A^{(p)}_n}^\dagger P_n A^{(p)} ( \frac{ \sin j t}{\sqrt{ \pi }} )(t)
 = \theta^{(p)}_0 \frac{1}{\sqrt{2\pi}} + \sum^n_{k=1} \theta^{(p)}_k \frac{ \cos k t}{\sqrt{ \pi }} + \sum^n_{k=1} \omega^{(p)}_k \frac{ \sin k t}{\sqrt{ \pi }}.
 \end{equation*}
 When $ j \geq n+1 $,
 \begin{equation*}
 \theta^{(p)}_0
 = \theta^{(p)}_0 (n,j) = (({A^{(p)}_n}^\dagger P_n A^{(p)} ( \frac{ \sin j t}{\sqrt{ \pi }} ), \frac{1}{\sqrt{2\pi}}  )_{L^2},
 \end{equation*}
  \begin{equation*}
 \theta^{(p)}_k
 = \theta^{(p)}_k (n,j) = (({A^{(p)}_n}^\dagger P_n A^{(p)} ( \frac{ \sin j t}{\sqrt{ \pi }} ), \frac{ \cos k t}{\sqrt{ \pi }} )_{L^2},
 \end{equation*}
   \begin{equation*}
 \omega^{(p)}_k
 = \omega^{(p)}_k (n,j) = (({A^{(p)}_n}^\dagger P_n A^{(p)} ( \frac{ \sin j t}{\sqrt{ \pi }} ), \frac{ \sin k t}{\sqrt{ \pi }} )_{L^2},
 \end{equation*}
and
\begin{equation}
 \vert \theta^{(p)}_0 \vert \leq   \frac{C_4^{(p)}}{j} , \ \vert \theta^{(p)}_k \vert \leq  \frac{C_5^{(p)}}{j} , \vert \omega^{(p)}_k \vert \leq \frac{C_6^{(p)}}{j}, \   1 \leq k \leq  n.
  \end{equation}
where
\begin{eqnarray*}
C_4^{(1)}=\frac{\sqrt{2}}{\pi} , \  C_4^{(2)}=\frac{3\sqrt{2}}{2}, \  C_4^{(3)}= \frac{44\sqrt{2}}{3}, \\
C_5^{(1)}= \frac{2}{\pi}, \ C_5^{(2)}= 3, \  C_5^{(3)}=30, \\
C_6^{(1)}=0, \  C_6^{(2)}= 5, \ C_6^{(3)}=48.
\end{eqnarray*}
Notice that when $ p =3 $, we need the extra condition $ n \geq 5 $ to maintain above estimate.
\end{lemma}
\begin{proof}\indent
 \textbf{Case $ p=1 $}:
 When \begin{math} j \geq n+1 \end{math}, insert (B.4) into (4.1),(4.2),(4.3), then
\begin{equation*}
 ({A^{(1)}_n}^\dagger P_n A^{(1)} ( \frac{\sin j t}{\sqrt{\pi}} ))(t)
= {A^{(1)}_n}^\dagger  (   \frac{\sqrt{2} }{ j } \cdot \frac{1}{\sqrt{2\pi}}     )
= \frac{\sqrt{2}}{ \pi j } \cdot \frac{1}{\sqrt{2\pi}}+ \sum^n_{ k=1 } \frac{2}{ \pi j } \cdot \frac{\cos k t}{\sqrt{\pi}} .
\end{equation*}
This gives lemma for case $ p=1 $.
\newline \textbf{Case $ p=2 $}:
Insert (B.5) into (4.4), and it follows that
\begin{equation*}
\theta^{(2)}_0 = \frac{1}{2n+1}\frac{  \sqrt{2}   }{ 4\pi j }L^{-1}_n.
\end{equation*}
With Proposition C.1 (C.1), it follows that
\begin{equation}
0 \leq \theta^{(2)}_0 \leq \frac{3\sqrt{2}}{ 2 j }.
\end{equation}
Besides, insert (B.5) into (4.5),(4.6), then
\begin{equation*}
\theta^{(2)}_k =\sqrt{2} \theta^{(2)}_0, \
\omega^{(2)}_k =  \frac{1}{j} \frac{ 2k }{ 2n+1 } -\frac{k}{2n+1} \cdot \sqrt{2} \pi \theta^{(2)}_0.
  \end{equation*}
Then by (5.7) we have
  \begin{equation*}
 0 \leq \theta^{(2)}_k \leq  \frac{3}{  j }, \ -\frac{5}{j} \leq \omega^{(2)}_k \leq  \frac{1}{j}.
\end{equation*}
 \textbf{Case $ p=3 $}:
Insert (B.6) into (4.7), and it follows that
\begin{equation*}
\theta^{(3)}_0 = \frac{T^{-1}_n}{4\pi^3}\frac{1}{j}(F_n  - \frac{\sqrt{2}}{j^2}).
\end{equation*}
Notice that it is easy to obtain that
\begin{equation*}
 \vert F_n  - \frac{\sqrt{2}}{j^2} \vert \leq   \frac{4 \sqrt{2}}{n(2n+1)}
\end{equation*}
from Proposition C.1 (C.3).
In this way, with Proposition C.1 (C.4), when $ n \geq 5 $,
\begin{equation*}
 \vert \theta^{(3)}_0 \vert = \frac{T^{-1}_n}{4\pi^3}\frac{1}{j}\vert F_n  - \frac{\sqrt{2}}{j^2} \vert
\leq 396 n(2n+1)\cdot \frac{1}{4\pi^3} \frac{1}{j}\cdot  \frac{4 \sqrt{2}}{n(2n+1)}
\end{equation*}
\begin{equation}
\leq \frac{396\sqrt{2}}{27} \frac{1}{j} =\frac{44\sqrt{2}}{3} \frac{1}{j}.
\end{equation}
Besides, insert (B.6) into (4.8), and we have
\begin{equation*}
 \theta^{(3)}_k = \frac{1}{(2n+1)^2} \frac{2\pi k^2}{j} + \sqrt{2} (1+ \frac{2k^2}{2n+1} K_n) \theta^{(3)}_0.
\end{equation*}
Hence, by (5.8)
\begin{equation*}
 \vert \theta^{(3)}_k \vert \leq  \frac{4k^2}{(2n+1)^2} \frac{\frac{\pi}{2}}{j} + \sqrt{2} \theta^{(3)}_0
 \leq   \frac{2}{j} + \sqrt{2} \theta^{(3)}_0 \leq \frac{2}{j} + \sqrt{2} \cdot \frac{44\sqrt{2}}{3} \frac{1}{j} = \frac{30}{j}.
\end{equation*}
Further, insert (B.6) into (4.9), and it follows that
  \begin{equation*}
\omega^{(3)}_k =  \frac{2k}{2n+1} \frac{1}{j} - \frac{ k }{ 2n+1 } \sqrt{2} \pi \theta^{(3)}_0.
  \end{equation*}
  Hence, by (5.8)
  \begin{equation*}
\vert \omega^{(3)}_k \vert \leq  \frac{1}{j} + \frac{ \sqrt{2}}{2} \pi \vert \theta^{(3)}_0 \vert
 \leq   \frac{1}{j} + \frac{ \sqrt{2}}{2} \pi \frac{44\sqrt{2}}{3} \frac{1}{j} \leq \frac{48}{j}.
  \end{equation*}
\end{proof}
\begin{lemma}
Set $ A^{(p)}, A_n^{(p)} $ defined in (2.4), (3.3) respectively. Then
\begin{equation*}
 \Vert K^{(p)}_n \Vert_{L^2 \to L^2} \leq \kappa^{(p)}, \ \forall n \in \mathrm{N}
\end{equation*}
where
\begin{equation*}
K^{(p)}_n := {A_n^{(p)}}^\dagger P_n A^{(p)}(I - P_n): L^2(0, 2\pi) \longrightarrow X_n \subseteq L^2(0, 2\pi)
\end{equation*}
\end{lemma}
\begin{remark}
With direct computations, we can obtain that
  \begin{equation*}
  \kappa^{(1)} \approx 0.7801, \    \kappa^{(2)} \approx 7.3729, \  \kappa^{(3)}  \approx 74.8198.
  \end{equation*}
\end{remark}
\begin{proof}
Set $ v =a_0 \frac{1}{\sqrt{2\pi}} +\sum^\infty_{k=1} a_k \frac{\cos k t}{\sqrt{\pi}} +\sum^\infty_{k=1} b_k \frac{\sin k t}{\sqrt{\pi}} $
such that $ \Vert v \Vert_{L^2} = 1 $, that is, $ a^2_0 + \sum^\infty_{k=1} a^2_k + \sum^\infty_{k=1} b^2_k=1. $
We consider the estimate on $ \Vert K^{(p)}_n v \Vert_{L^2}, \ n \in \mathrm{N} $.
\newline \indent Since $ {A^{(p)}_n}^\dagger P_n A^{(p)}  $ $ (n \in \mathrm{N})$ is continuous with Remark 4.2,
\begin{equation*}
K^{(p)}_n v = {A^{(p)}_n}^\dagger P_n A^{(p)} (I-P_n) v
\end{equation*}
\begin{equation*}
= {A^{(p)}_n}^\dagger P_n A^{(p)}  ( \sum^\infty_{j=n+1} a_j \frac{\cos j t}{\sqrt{\pi}} +\sum^\infty_{j=n+1} b_j \frac{\sin j t}{\sqrt{\pi}})
\end{equation*}
\begin{equation*}
=\sum_{j=n+1}^\infty a_j {A^{(p)}_n}^\dagger P_n A^{(p)} ( \frac{\cos j t}{\sqrt{\pi}})  + \sum_{j=n+1}^\infty b_j {A^{(p)}_n}^\dagger P_n A^{(p)}  ( \frac{\sin j t}{\sqrt{\pi}}   ) .
\end{equation*}
Recall Lemma 5.1 and Lemma 5.2. It follows that
\begin{equation}
 {A^{(p)}_n}^\dagger P_n A^{(p)}(I-P_n) v
= H_0 \frac{1}{\sqrt{2\pi}} + \sum^n_{k=1} H_k  \frac{\cos k t}{\sqrt{\pi}}
+ \sum^n_{k=1} G_k \frac{\sin k t}{\sqrt{\pi}},
 \end{equation}
 where
 \begin{equation*}
 H_0 = \sum^\infty_{j=n+1}(a_j \alpha^{(p)}_0(n,j) +b_j \theta^{(p)}_0(n,j) ),
 \end{equation*}
  \begin{equation*}
 H_k = \sum_{j=n+1}^\infty( a_j \alpha^{(p)}_k(n,j) +b_j \theta^{(p)}_k(n,j)),
 \end{equation*}
   \begin{equation*}
 G_k = \sum_{j=n+1}^\infty(  a_j \beta^{(p)}_k(n,j)+ b_j \omega^{(p)}_k(n,j)).
 \end{equation*}
By (5.2), (5.6) and the Cauchy inequality, we have
   \begin{equation}
H^2_0 \leq \frac{{C^{(p)}_1}^2+{C^{(p)}_4}^2}{n} \sum^\infty_{j=n+1} (a^2_j+b^2_j),
\end{equation}
\begin{equation}
H^2_k\leq  \frac{{C^{(p)}_2}^2+{C^{(p)}_5}^2}{n} \sum^\infty_{j=n+1} (a^2_j+b^2_j),
\end{equation}
 \begin{equation}
 G^2_k \leq \frac{{C^{(p)}_3}^2+{C^{(p)}_6}^2}{n} \sum^\infty_{j=n+1} (a^2_j+b^2_j).
 \end{equation}
 (5.10),(5.11),(5.12) together with (5.9) give that, for all $ v $ such that $ \Vert v \Vert_{L^2} = 1 $,
\begin{equation*}
  \Vert {K^{(p)}_n} v \Vert_{L^2}^2 \leq ( \sum^6_{i=1} {C^{(p)}_i}^2) \sum^\infty_{j=n+1} (a^2_j+b^2_j)
  \end{equation*}
\begin{equation*}
\leq (\sum^6_{i=1} {C^{(p)}_i}^2) \Vert v \Vert_{L^2}^2  \  \  (\kappa^{(p)}:=  \sqrt{\sum^6_{i=1} {C^{(p)}_i}^2}).
  \end{equation*}
  where $ C^{(p)}_i \ (p=1,2,3; \ i = 1,2, \cdots, 6) $
are all constants defined in Lemma 5.1 and Lemma 5.2.
\end{proof}

  \begin{theorem}
The Groetsch regularity holds for Galerkin setting (3.2); that is,
\begin{equation*}
\sup_n \Vert R^{(p)}_n \Vert_{L^2 \longrightarrow L^2} \leq \gamma^{(p)}   < \infty \ (\gamma^{(p)} := 1 + \kappa^{(p)}),
\end{equation*}
where
\begin{equation*}
 R^{(p)}_n := {A_n^{(p)}}^\dagger P_n A^{(p)} : L^2(0, 2\pi) \longrightarrow X_n \subseteq L^2(0, 2\pi)
\end{equation*}
and $ A^{(p)}, A_n^{(p)} $ are defined in (2.4), (3.3) respectively.
\end{theorem}
\begin{proof}
Since
\begin{equation*}
{A^{(p)}_n}^\dagger P_n A^{(p)}  = K^{(p)}_n +  {A^{(p)}_n}^\dagger P_n A^{(p)} P_n
= K^{(p)}_n +  {A^{(p)}_n}^\dagger {A^{(p)}_n}
 \end{equation*}
\begin{equation*}
= K^{(p)}_n + P_{\mathcal{N}(A^{(p)}_n)^{\perp_n}} = K^{(p)}_n  +  P_n \ (\textrm{Since} \ \mathcal{N}(A^{(p)}_n)={0} \ \textrm{in Remark 4.1}),
\end{equation*}
by Lemma 5.3 we have
  \begin{equation*}
\Vert {A^{(p)}_n}^\dagger P_n A^{(p)} \Vert_{L^2 \to L^2 } \leq \gamma^{(p)},  \  n \in \mathrm{N}, \
\gamma^{(p)} := 1+\kappa^{(p)}.
\end{equation*}
\end{proof}
After the examination of (5.1), we have an estimate on the approximation error.
\begin{corollary}
For $  {A^{(p)}_n} $ defined as (3.3), we have
\begin{equation*}
  \Vert {A^{(p)}_n}^\dagger P_n y - y^{(p)} \Vert_{L^2} \leq (\gamma^{(p)}+1) \Vert (I- P_n) y^{(p)} \Vert_{L^2} \longrightarrow 0 \ (n \to \infty)
  \end{equation*}
  for every $ y \in   {\mathcal{H}}^p_0(0,2\pi) $. Furthermore, with a priori information $ y^{(p)} \in H_{per}^l (0, 2\pi) $, it yields that
  \begin{equation*}
   \Vert {A^{(p)}_n}^\dagger P_n y - y^{(p)} \Vert_{L^2} \leq \frac{(\gamma^{(p)}+1)}{n} \Vert y^{(p)} \Vert_{H_{per}^l},
  \end{equation*}
where $ \gamma^{(p)} $ is constant given in Theorem 5.1.
\end{corollary}
\begin{proof}
By Lemma 2.2, 2,3,   for $ y \in   {\mathcal{H}}^p_0(0,2\pi) $,
\begin{equation*}
   \Vert {A^{(p)}_n}^\dagger P_n y - y^{(p)} \Vert_{L^2}
= \Vert {A^{(p)}_n}^\dagger P_n P_{\overline{\mathcal{R}(A^{(p)})}}  y - {A^{(p)}}^\dagger y \Vert_{L^2}.
\end{equation*}
\indent Using Lemma 2.4, Remark 4.1, Theorem 5.1, it yields that
 \begin{equation}
   \Vert {A^{(p)}_n}^\dagger P_n y - y^{(p)} \Vert_{L^2}
 \leq  (\gamma^{(p)}+1) \Vert (I-P_n) y^{(p)} \Vert_{L^2}.
 \end{equation}
 Now provided with a priori information $ y^{(p)} \in H_{per}^{l}(0,2\pi), \ l > 0. $ It yields that
\begin{equation*}
 \Vert (I-P_n) y^{(p)} \Vert_{L^2} \leq \frac{1}{n^l} \Vert y^{(p)} \Vert_{H_{per}^l}
 \end{equation*}
 from Lemma 2.5. This improves (5.13) to the latter result needed.
\end{proof}
\begin{corollary}
For $  {A^{(p)}_n} $ defined as (3.3), we have
\begin{equation*}
  \Vert {A^{(p)}_n}^\dagger P_n y\Vert_{L^2} \longrightarrow \infty \ (n \to \infty)
  \end{equation*}
  for every $ y \in L^2 (0,2\pi) \setminus  \mathcal{R}(A^{(p)}) $, where $ \mathcal{R}(A^{(p)}) = {\mathcal{H}}^p_0(0,2\pi) $.
\end{corollary}
\begin{proof}
Lemma 2.4 tells that $ \mathcal{D}({A^{(p)}}^\dagger) = \mathcal{R}(A^{(p)}) $.
Since the estimate of (5.1) holds, with Lemma 2.4 (b), the result surely holds.
\end{proof}
Here Corollary 5.2 tells us two questions:
\begin{itemize}
\item the first question is, for $ p $ order numerical differentiation, when $ y \in {\mathcal{H}}^p_0(0,2\pi)$, with interfuse of noise $ \delta y $, $ y^\delta $ would generally locate in $ L^2 (0,2\pi) \setminus \mathcal{H}^p_0 (0,2\pi) $. Then with increasing choice of index $ n $ independent of noise level $ \delta $,
\begin{equation*}
  \Vert {A^{(p)}_n}^\dagger P_n y^\delta \Vert_{L^2} \longrightarrow \infty \ (n \to \infty)
  \end{equation*}
  This fact shows that, without proper parameter choice strategy for $ n^{(p)} := n^{(p)}(\delta) $, numerical scheme constructed as $ {A^{(p)}_n}^\dagger P_n y^\delta $ is natively instable.
  \item the second question is a worse one. With the more general setting  $ y \in H^p( 0, 2\pi) $ for $ p $ order Numerical differentiation, if $ y \in H^p(0,2\pi) \setminus {\mathcal{H}}^p_0(0,2\pi) $, then, with any parameter choice strategy $ n^{(p)} := n^{(p)}(\delta) $ such that $ n^{(p)}(\delta) \to +\infty \ (\delta \to 0^+) $, approximation error
  \begin{equation*}
 e^{(p)}_A = \Vert {A^{(p)}_{{n^{(p)}}(\delta)}}^\dagger P_{{n^{(p)}}(\delta)} y - y^{(p)} \Vert_{L^2}
    \longrightarrow \infty \ (\delta \to 0^+).
  \end{equation*}
  In addition with estimate on noise error $ e^{(p)}_N \leq C^{(p)} n^{(p)}(\delta) \to \infty \ (\delta \to 0^+) $ (by (4.10)), one could see the invalidness of single regularization parameter choice since only adjusting parameter choice $ n = n^{(p)}(\delta) $ can not gives $ e^{(p)}_T \to 0 \ (\delta \to 0^+) $.
  \end{itemize}
  The following two sections will answer above two questions respectively.
\section{Total Error Estimate for $ y \in {\mathcal{H}}^p_0(0,2\pi) $ and Parameter Choice for Regularization}
To solve the first question, we introduce regularization in the following procedure:
\newline \indent Combining Corollary 5.1 with Remark 4.2, it gives
\begin{theorem}
Set $  {A^{(p)}_n} $ as (3.3), $ y \in {\mathcal{H}}^p_0(0,2\pi) $. Then
\begin{equation}
\Vert {A^{(p)}_n}^\dagger P_n y^\delta - y^{(p)}\Vert_{L^2(0,2\pi)}
\leq C^{(p)} n^p \delta + \Vert (I-P_n) y^{(p)} \Vert_{L^2}.
\end{equation}
Furthermore, with a priori information $ y^{(p)} \in H_{per}^{l}(0,2\pi) $,
\begin{equation}
\Vert {A^{(p)}_n}^\dagger P_n y^\delta - y^{(p)}\Vert_{L^2(0,2\pi)}
\leq C^{(p)} n^p \delta + (\gamma^{(p)}+1)\frac{1}{n^{l}} \Vert y^{(p)} \Vert_{H_{per}^{l}}.
\end{equation}
\end{theorem}
\begin{remark}
In the case that $ y \in {\mathcal{H}}^p_0(0,2\pi) $ is provided but no a priori information on exact solution $ y^{(p)} $,
we determine parameter choice strategy from (6.1) as
\begin{equation}
 n_1^{(p)} := n_1^{(p)} (\delta) = \kappa \delta^{a- \frac{1}{p}},
\end{equation}
where $ a \in (0, \frac{1}{p}) $ is optional. However, we specify that, in this case, the convergence rate can not be obtained higher than $ O(1) $.
\newline \indent In the case that $ y \in {\mathcal{H}}^p_0(0,2\pi) $ and $ y^{(p)} \in H^l_{per} (0,2\pi) $,
we could determine parameter choice strategy from (6.2) as
\begin{equation}
n_2^{(p)} = n_2^{(p)}(\delta) = (\frac{l (\gamma^{(p)}+1) \Vert y^{(p)} \Vert_{H_{per}^l}}{p C^{(p)}})^{\frac{1}{l+p}} \delta^{-\frac{1}{l+p}}.
\end{equation}
Hence it follows that
\begin{equation}
\Vert {A^{(p)}_{n_2^{(p)}(\delta)}}^\dagger P_{n_2^{(p)}(\delta)} y^\delta - y^{(p)} \Vert_{L^2(0,2\pi)}
\leq  \Gamma_p \Vert y^{(p)} \Vert^{\frac{p}{l+p}}_{H_{per}^l} \delta^{\frac{l}{l+p}},
\end{equation}
where
\begin{equation*}
\Gamma_p := ((\frac{l}{p})^{\frac{p}{l+p}}+(\frac{l}{p})^{\frac{-1}{l+p}}) (C^{(p)})^{\frac{l}{l+p}}(\gamma^{(p)}+1)^{\frac{p}{l+p}}.
\end{equation*}
\end{remark}
\begin{remark}
Assume that
\begin{equation}
 y^{(p)} \in H_{per}^{2 \mu p } (0,2\pi) \ (\mu = \frac{1}{2} \ \textrm{or}  \ 1).
 \end{equation}
 Choosing $  n_3^{(p)}=n_3^{(p)} (\delta)=\delta^{-\frac{1}{(2\mu + 1 )p}} $, we gain the optimal convergence rate from (6.5), that is,
\begin{equation*}
 \Vert {A^{(p)}_{n_3^{(p)}(\delta)}}^\dagger P_{n_3^{(p)}(\delta)} y^\delta - y^{(p)} \Vert = O( \delta^{\frac{2 \mu}{2\mu+1}}) .
 \end{equation*}
Here we specify that (6.6) is a slightly variant version of the standard source condition stated in [10], that is, $  y^{(p)} \in \mathcal{R}({A^{(p)}}^* A^{(p)})^{\mu} \ ( \mu = \frac{1}{2} \ \textrm{or}  \ 1 ) $, where
\begin{equation*}
\mathcal{R}({A^{(p)}}^* A^{(p)})^{\frac{1}{2}}= \mathcal{R}({A^{(p)}}^*) = \mathcal{H}^p_{2\pi} (0,2\pi), \
\mathcal{R}({A^{(p)}}^* A^{(p)}) = {\dot \mathcal{H}}^{2p} (0,2\pi).
\end{equation*}
Notice that
\begin{equation*}
Codim_{H^p} \mathcal{H}^p_{2\pi} (0,2\pi) = Codim_{H^p}  H_{per}^{p} (0,2\pi) = p
 \end{equation*}
 and
 \begin{equation*}
 Codim_{H^{2p}}  {\dot \mathcal{H}}^{2p} (0,2\pi) = Codim_{H^{2p}}  H_{per}^{2p} (0,2\pi) = 2p.
\end{equation*}
\end{remark}
\begin{remark}
Assume that noise level $ \delta $ range in any closed interval $ [\delta_0, \delta_1] \subseteq (0,+\infty) $.
 If $ y \in {\mathcal{H}}^p_0(0,2\pi) $ and $ y^{(p)} \in H_{per}^\infty(0,2\pi) $ such that $ \lim_{l \to \infty} \Vert y^{(p)} \Vert^{\frac{1}{l+p}}_{H_{per}^l} $ exists, then the regularization parameter is determined by (6.4) as
 \begin{equation}
 n^{(p)} = \lim_{l \to \infty} \Vert y^{(p)} \Vert^{\frac{1}{l+p}}_{H_{per}^l},
\end{equation}
which only depends on exact derivative $ y^{(p)} $, not concerned with noise level $ \delta $. Furthermore,
in the case that $ y \in \mathcal{H}^p_0 (0,2\pi)$ and
\begin{equation}
y^{(p)} = a_0 + \sum^{N_1}_{k=1} b_k \cos k t + \sum^{N_2}_{k=1} c_k \sin k t  \in H_{per}^\infty(0,2\pi),
\end{equation}
where $ b_{N_1}, \ c_{N_2} \neq 0 $. The regularization parameter is determined by (6.7) as
 \begin{equation}
 n^{(p)} = \max(N_1,N_2),
\end{equation}
which only depends on the highest frequency of trigonometric polynomial in (6.8), not concerned with noise level $ \delta $.
\end{remark}
\section{Extended Numerical Differentiation on $ H^p(0,2\pi) $}
\subsection{Extended result with exact measurements at endpoint $ x = 0 $ }
Theorem 6.1 provides a result of stable numerical differentiation
on $ y \in  {\mathcal{H}}^p_0(0,2\pi) $, where
\begin{equation*}
{\mathcal{H}}^p_0(0,2\pi) :=\{ y \in H^p(0,2\pi), y(0)=\cdots = y^{(p-1)}(0) = 0 \}.
\end{equation*}
We consider to remove the restriction on initial value data, and extend the result into case $ y \in H^p(0,2\pi) $.
\newline \indent Observing that, for $ y \in H^p(0,2\pi) $,
\begin{equation*}
 y(x)-\sum^{p-1}_{k=0} \frac{y^{(k)}(0)}{k!} x^k \in {\mathcal{H}}^p_0(0,2\pi),
\end{equation*}
we naturally adjust regularized scheme (1.3) into
\begin{equation*}
{A^{(p)}_n}^\dagger P_n (y^\delta-\sum^{p-1}_{k=0} \frac{y^{(k)}(0)}{k!} x^k).
\end{equation*}
Now, given exact measurements on the initial value data,
\begin{equation*}
 y(0), y'(0), \cdots, y^{(p-1)} (0).
 \end{equation*}
we can adjust Theorem 6.1 into the following version.
\begin{theorem}
Set $  {A^{(p)}_n} $ as (3.3), $ y \in H^p(0,2\pi) $. Then
\begin{equation*}
\Vert {A^{(p)}_n}^\dagger P_n (y^\delta-\sum^{p-1}_{k=0} \frac{y^{(k)}(0)}{k!} x^k) - y^{(p)}\Vert_{L^2} \leq C^{(p)} n^p \delta + (\gamma^{(p)}+1) \Vert (I-P_n) y^{(p)} \Vert_{L^2}.
\end{equation*}
Furthermore, with a priori information $ y^{(p)} \in H_{per}^{l}(0,2\pi) $,
\begin{equation*}
\Vert {A^{(p)}_n}^\dagger P_n (y^\delta-\sum^{p-1}_{k=0} \frac{y^{(k)}(0)}{k!} x^k) - y^{(p)}\Vert_{L^2} \leq C^{(p)} n^p \delta + (\gamma^{(p)}+1)\frac{1}{n^{l}} \Vert y^{(p)} \Vert_{H_{per}^{l}}.
\end{equation*}
\end{theorem}
\subsection{Extended result with noisy measurements at endpoint $ x = 0 $}
However, in practical cases, one can not obtain initial value data
 $ y(0), y'(0),y''(0) $ exactly. Instead, one could only obtain a cluster of noisy data,
 denoted as $ \Lambda_0 (0), \Lambda_1 (0),\Lambda_2 (0) $ respectively. Now
provided with above endpoint measurement, we reformulate the problem
of $ p $ order numerical differentiation as:
\begin{problem}
 Assume that we have
\begin{itemize}
\item $ y \in H^p(0,2\pi) $ and $ y^\delta $ measured on $ ( 0, 2\pi ) $,
which belongs to $ L^2(0,2\pi)$ such that $ \Vert y^\delta - y \Vert_{L^2} \leq \delta $,
\item Noisy initial value data $ \Lambda_0 (0), \  \cdots, \Lambda_{p-1} (0) $
for $ y(0), \cdots, y^{(p-1)}(0) $ respectively, which satisfies that
\begin{equation*}
 \vert \Lambda_k (0) - y^{(k)}(0) \vert \leq \delta_i , \ k= 0,\cdots, p-1.
 \end{equation*}
 \end{itemize}
 How to gain stable approximation to $ y^{(p)} $?
 \end{problem}
An estimate similar to (6.2) is constructed to answer this question:
\begin{theorem}
Set $  {A^{(p)}_n} $ as (3.3), $ y \in H^p(0,2\pi) $ and $ y^{(p)} \in H_{per}^{l}(0,2\pi) $. Then
\begin{equation*}
\Vert {A^{(p)}_n}^\dagger P_n (y^\delta-(\sum^{p-1}_{k=0} \frac{\Lambda_k(0)}{k!} x^k )) - y^{(p)}\Vert_{L^2}
\end{equation*}
\begin{equation*}
\leq
C^{(p)} n^p \delta +\Delta_p C^{(p)} n^p \delta_i + (\gamma^{(p)}+1)\frac{1}{n^{l}} \Vert y^{(p)} \Vert_{H_{per}^{l}}
\end{equation*}
For convenience of notations, we set $ \delta_i = \delta $, then it follows that
\begin{equation*}
\Vert {A^{(p)}_n}^\dagger P_n (y^\delta-(\sum^{p-1}_{k=0} \frac{\Lambda_k(0)}{k!} x^k )) - y^{(p)}\Vert_{L^2}
\end{equation*}
\begin{equation*}
\leq
C_{\Delta}^{(p)} n^p \delta  + (\gamma^{(p)}+1)\frac{1}{n^{l}} \Vert y^{(p)} \Vert_{H_{per}^{l}},
\end{equation*}
where $ C_{\Delta}^{(p)} := (\Delta_p +1 ) C^{(p)}  $ and $ \Delta_p := \sum\limits^{p-1}_{k=0} \frac{ \Vert  x^k \Vert_{L^2} }{k!} $.
\end{theorem}
\begin{remark}
In this case, it is necessary to specify that the parameter choice strategy should be adjusted from (6.4) to the following,
\begin{equation}
n^{(p)} = n^{(p)}(\delta) = (\frac{l (\gamma^{(p)}+1) \Vert y^{(p)} \Vert_{H_{per}^l}}{p C_{\Delta}^{(p)}})^{\frac{1}{l+p}} \delta^{-\frac{1}{l+p}}.
\end{equation}
Also, assume that noise level $ \delta $ range in the interval $ (\delta_0, \delta_1) $, where $ 0 < \delta_0 <<1 $ and $ \delta_1 < +\infty $.  When $ y \in H^p(0,2\pi) $ and $ y^{(p)} \in H_{per}^\infty $ such that $ \lim_{l \to \infty} \Vert y^{(p)} \Vert^{\frac{1}{l+p}}_{H_{per}^l} $ exists, the regularization parameter is determined by (7.1) as
 \begin{equation}
 n^{(p)} = \lim_{l \to \infty} \Vert y^{(p)} \Vert^{\frac{1}{l+p}}_{H^l},
\end{equation}
which remains not concerned with noise level $ \delta$, also not concerned with the additional noise level $ \delta_i $
of initial value data. Besides, in the case that
\begin{equation*}
y = \sum^{p-1}_{k=0} a_k x^k + \sum^{N_1}_{k=1} b_k \cos k t + \sum^{N_2}_{k=1} c_k \sin k t,
\end{equation*}
where $ b_{N_1}, \ c_{N_2} \neq 0 $. The optimal parameter choice is determined by (7.2) as
 \begin{equation}
 n^{(p)} = \max(N_1,N_2),
\end{equation}
which is just the same as (6.9),  still not concerned with noise level $ \delta $  and additional noise in initial value data.
\end{remark}
\begin{remark}
The optimal convergence rate $ O(\delta^{\frac{2\mu}{2\mu+1}}) $ can be achieved in the same way as Remark 6.2.
\end{remark}
\begin{proof}
 For $ y \in H^p(0,2\pi) $ and
$ y^\delta \in L^2[0,2\pi] $ with $ \Vert y^\delta -y \Vert \leq \delta $,
\begin{equation*}
 \Vert {A^{(p)}_n}^\dagger P_n (y^\delta-(\sum^{p-1}_{k=0} \frac{\Lambda_k(0)}{k!} x^k )) - y^{(p)}\Vert_{L^2}
\leq e'_T + e'_P,
 \end{equation*}
 where
\begin{equation*}
e'_T:= \Vert {A^{(p)}_n}^\dagger P_n ( y^\delta-\sum^{p-1}_{k=0} \frac{y^{(k)}(0)}{k!} x^k) - (y-\sum^{p-1}_{k=0} \frac{y^{(k)}(0)}{k!} x^k)^{(p)}\Vert_{L^2},
\end{equation*}
\begin{equation*}
e'_P:= \Vert {A^{(p)}_n}^\dagger P_n (\sum^{p-1}_{k=0} \frac{y^{(k)}(0)}{k!} x^k-\sum^{p-1}_{k=0} \frac{\Lambda_k(0)}{k!} x^k)\Vert_{L^2}.
\end{equation*}
Apply Theorem 7.1, and it follows that
\begin{equation*}
e'_T
\leq C^{(p)} n^p \delta + (\gamma^{(p)}+1)\frac{1}{n^{l}} \Vert y^{(p )} \Vert_{H_{per}^{l}}.
\end{equation*}
Besides,
\begin{equation*}
e'_P
 \leq \Vert {A^{(p)}_n}^\dagger \Vert \Vert  P_n \Vert \sum^{p-1}_{k=0} \frac{\delta_e}{k!} \Vert  x^k\Vert_{L^2} \leq \Delta_p C^{(p)} n^p \delta_i,
\end{equation*}
where $ \Delta_p := \sum\limits^{p-1}_{k=0} \frac{ \Vert  x^k \Vert_{L^2} }{k!} $. Then we have
\begin{equation*}
 \Vert {A^{(p)}_n}^\dagger P_n (y^\delta-(\sum^{p-1}_{k=0} \frac{\Lambda_k(0)}{k!} x^k )) - y^{(p)}\Vert_{L^2}
 \end{equation*}
\begin{equation*}
\leq C^{(p)} n^p \delta +\Delta_p C^{(p)} n^p \delta_i + (\gamma^{(p)}+1)\frac{1}{n^{l}} \Vert y^{(p )} \Vert_{H_{per}^{l}}.
\end{equation*}
\end{proof}
\section{Numerical Experiments}
All experiments are performed in Intel(R) Core(TM) i7-7500U CPU @2.70GHZ 2.90 GHZ Matlab R 2017a.
For all experiments, the regularized solution is given by
 \begin{equation*}
 \varphi^{p, \delta, \delta_i}_n := {A^{(p)}_n}^\dagger P_n (y^\delta-\Lambda^{(p)}(x)),
 \end{equation*}
with regularization parameter choice $ n = n^{(p)} = n^{(p)} (\delta, \delta_i) (p=1,2,3)$, where
\begin{equation*}
\delta y =\delta \frac{\sin k x }{\sqrt{\pi}}, \ y^\delta (x) = y(x) + \delta y.
\end{equation*}
and
\begin{equation*}
\Lambda^{(p)}(x) = \sum^{p-1}_{k=0}\frac{\Lambda_k(0)}{k!} x^k \ \textrm{with} \ \Lambda_k (0) = y^{(k)}(0) + \delta_i, \ k \in \overline{0,1,\cdots, p-1}.
\end{equation*}
All experiments are divided into two cases:
 \begin{itemize}
 \item \textbf{Case I}: $ \delta \neq 0, \delta_i = 0 $, that is, high frequency noise $ \delta y $ and exact initial value data.
  \item \textbf{Case II}: $ \delta = \delta_i \neq 0 $, that is, high frequency noise $ \delta y $ and noisy initial value data.
      \end{itemize}
\indent The following index is introduced to measure the computational accuracy in tests:
\begin{itemize}
\item Relative error
\begin{equation*}
 r=\frac{\Vert \varphi^{p, \delta, \delta_i}_{n^{(p)}(\delta, \delta_i)} - y^{(p)} \Vert_{L^2} }{ \Vert y^{(p)} \Vert_{L^2} }.
 \end{equation*}
\end{itemize}
\subsection{On smooth functions}
\begin{example}
Set
\begin{equation*}
p(x) = \sum^6_{k=1} \frac{1}{k^2} \sin(kx), \ q_i(x) = \sum^{p-1}_{i=0} 1*x^i, \ p=1,2,3
\end{equation*}
\begin{equation*}
y_i (x) = p(x) + q_i(x), \ y_i^\delta (x) = y_i (x) + \delta \frac{\sin 12 x}{\sqrt{\pi}}.
\end{equation*}
We use the $ y^\delta_i $ as test function for $ i $ order numerical differentiation. Notice that
\begin{equation*}
y_1' (x) = \sum^6_{k=1} \frac{1}{k} \cos(kx), \ y_2''(x) =  - \sum^6_{k=1} \sin(kx) ,
\end{equation*}
\begin{equation*}
 y_3'''(x) = - \sum^6_{k=1} k \cos(kx) .
\end{equation*}
\begin{equation*}
y_1 (0) = 1, \ y_2 (0) = 1, \ y'_2 (0) = 1+\sum^6_{k=1} \frac{1}{k}
\end{equation*}
\begin{equation*}
 y_3(0) = 1, \ y'_3 (0) = 1+\sum^6_{k=1} \frac{1}{k}, \  y''_3 (0) = 2.
\end{equation*}
\end{example}

\begin{table}
{\begin{tabular}{cccccccc}
\hline
 & & $ n $ & $ 2 $ & $ 4 $ & $ 6 $ & $ 8 $& 12  \\
\cline{2-8}
 $ p=1 $ & $ \delta_i = 0 $  & $ r $ & 0.4023 & 0.2132 & $1.5194e^{-16}$ & $ 1.5194e^{-16}$  & 0.0554 \\
\cline{2-8}
 & $ \delta_i = 0.01 $ & $ r $ & 0.4024 &  0.2035 & $ 0.0133 $ & $ 0.0152 $ & 0.0532 \\
 \cline{2-8}
 $ p=2 $ & $ \delta_i = 0 $  & $ r $ &  1.0695 &  0.6808 & $ 7.5244 e^{-15}$ & $ 1.1876e^{-14}$ & 0.3666 \\
\cline{2-8}
 & $ \delta_i = 0.01 $ & $ r $ &1.0830 & 0.7046 & $ 0.0732 $ & 0.1051 &0.3046 \\
 \cline{2-8}
 $ p=3 $ & $ \delta_i = 0 $  & $ r $ &1.1879   &  1.3802 &  $ 6.6497e^{-14} $  & $  1.6561e^{-13} $  & 0.0469 \\
\cline{2-8}
 & $ \delta_i = 0.01 $ & $ r $ & 1.1884 &  1.3827 & $ 0.0081 $ &$   0.0150 $ & 0.0355 \\
\hline
\end{tabular}}
\caption{Above experiments  correspond to the three examples in Example 8.1. Case I,II are uniformly set as $ (\delta, \delta_i) = (0.01, 0 ) $ and $ ( \delta, \delta_i ) = ( 0.01, 0.01) $ respectively. Notice that $ r $ denotes the relative error. }
\end{table}

\begin{table}
{\begin{tabular}{cccccccc}
\hline
 & & $ n $ & $ 2 $ & $ 4 $ & $ 6 $ & $ 8 $& 12  \\
\cline{2-8}
 $ p=1 $ & $ \delta_i = 0 $  & $ t $ & 0.3341 & 0.3673 & 0.4178 & 0.4577  & 0.5983 \\
\cline{2-8}
 & $ \delta_i = 0.01 $ & $ t $ & 0.3391 & 0.3728 & 0.4488 & 0.4935 & 0.5748 \\
 \cline{2-8}
 $ p=2 $ & $ \delta_i = 0 $  & $ t $ & 0.4670 & 0.6070 &  0.6385  & 0.8051  & 0.8400 \\
\cline{2-8}
 & $ \delta_i = 0.01 $ & $ t $ & 0.4892  & 0.5641 & 0.6218 &0.7126 & 0.8160 \\
  \cline{2-8}
 $ p=3 $ & $ \delta_i = 0 $  & $ t $ &  0.2755 &  0.3958 & 0.4497 & 0.5024 & 0.6810 \\
\cline{2-8}
 & $ \delta_i = 0.01 $ & $ t $ & 0.4723 & 0.5151 & 0.6348 & 0.7240 & 1.1838 \\
\hline
\end{tabular}}
\caption{$ t $ denotes the CPU time $ (s) $ for the corresponding experiment in Table 1. }
\end{table}
\subsubsection{Unified observation on cases with smooth derivative}
We first investigate into the case  $ (\delta, \delta_i) = ( 0.01, 0 ) $. All data in this case can be divided into three phases  $ \mathcal{V}_1 =\{ 2,4 \} $, $ \mathcal{V}_2 =\{ 6,8 \} $, $ \mathcal{V}_3 =\{ 12 \} $.
\newline \indent We can compare above three phases and quickly find that when $ n \in \mathcal{V}_2 $, the relative error $ r $ is the least and almost approaches $ 0 $. This displays the good filtering effect of algorithm on specific class of functions (the sum of trigonometric polynomial and polynomial of order less than order $ p $, where $ p $ denotes the order of numerical differentiation), and also correspond to the fact indicated by (6.9): the best regularization parameter $ n = 6 $.
\newline \indent Now we explain the source of the good filtering effect in case $ p = 2 $, the other cases are similar.
\begin{itemize}
\item The exact measurement of initial value data help give a precise Taylor polynomial truncation to eliminate the polynomial term $ q_2(x) $ in $ y^{\delta}_2 $, thus the computational scheme is transformed into $ {A^{(2)}_n}^\dagger P_n (\bar p(x) + \delta y ) $, where $ \bar p(x) = p(x) - \sum^6_{k=1} \frac{1}{k} x $.
\item The parameter choice $ n =6 $ appropriately eliminate the noise component $ \delta y $, now $ {A^{(2)}_6}^\dagger P_6 (\bar p(x) + \delta y ) = {A^{(2)}_6}^\dagger P_6 \bar p(x) $. Notice that $ \bar p(x) \in \mathcal{R}(A^{(2)}) $, $ {A^{(2)}}^\dagger \bar p(x) = y''_2 $,
    \begin{equation*}
    A^{(2)}_6 y''_2 = P_6 A^{(2)} P_6 y''_2 = P_6 A^{(2)} y''_2
    \end{equation*}
    \begin{equation*}
    = P_6 A^{(2)} {A^{(2)}}^\dagger \bar p(x)  = P_6 P_{\overline{\mathcal{R}(A^{(2)})}} \bar p(x) = P_6 \bar p(x) \ \textrm{by} \ (2.1).
    \end{equation*}
Then the uniqueness of Galerkin system in Theorem 4.1 gives that $ {A^{(2)}_6}^\dagger P_6 \bar p(x) = y''_2 $, that is, approximate solution strictly equals to the exact solution in this case, the good filtering effect appears.
\end{itemize}
The deviation of the accuracy in $ \mathcal{V}_1$ and $ \mathcal{V}_3 $ can also be explained in the same way. The former case of choice $ n = 2 $ with lower accuracy is due to that $ P_n y^\delta_2, \ n =2 $ does not cover the major part of $ p(x) $, but with the increase of $ n \in \mathcal{V}_1 $, the coverage increases and hence the accuracy improves. As to the latter case $ \mathcal{V}_3 $, now the $ \delta y = \frac{\sin 12x }{\sqrt{\pi}} $ come into the computation, thus the good filtering effect disappears.
\newline \indent For case II with noise pair $ (\delta, \delta_i) = (0.01, 0.01) $, it can be seen in Table 1 that, when regularization parameter $ n \in \mathcal{V}_2 $, the corresponding relative error $ r $ all approach or exceed 0.01 uniformly. Now, compared to the case I with the same choice for regularization parameter, the good filtering effect of case I are strongly weakened. This is because the noise $ \delta_i $ in initial value data bring a not complete Taylor polynomial truncation and hence parameter $ n = 6 $ can not eliminate the lower-frequency noise components in
\begin{equation*}
 \Lambda^{(p)}(x) - \sum^{p-1}_{k=0} \frac{y^{(k)}(0)}{k!}x^k.
 \end{equation*}
 \subsection{On periodic weakly differentiable derivatives}
 In this subsection, we mainly investigate the effectiveness of parameter choice (6.4) and (7.1) for general case (compared with the case in subsection 8.1).
 \subsubsection{First order}
\begin{example}
Set
\begin{equation*}
 y(x) = \left\{
  \begin{array}{rcl}
   \pi x - \frac{1}{2}x^2 , \ 0 \leq x < \pi, \\
\frac{1}{2} x^2 - \pi x + \pi^2 , \pi \leq x < 2\pi .
\end{array} \right.
\end{equation*}
\begin{equation*}
y(0) = 0,
\end{equation*}
\begin{equation*}
y^\delta (x) = y(x) + \delta \frac{\sin  k x}{\sqrt{\pi}},
\end{equation*}
\begin{equation*} y' (x) = \left\{
\begin{array}{rcl}
 \pi - x , \ 0 \leq x < \pi, \\
x- \pi , \pi  \leq x < 2\pi .
\end{array} \in H^1_{per}(0,2\pi) \right  .
\end{equation*}
\end{example}
 \subsubsection{Second order}
\begin{example}
Set
\begin{equation*}
 y(x) = \left\{
  \begin{array}{rcl}
   \frac{1}{2} \pi x^2 - \frac{1}{6} x^3 , \ 0 \leq x < \pi, \\
\frac{1}{6} x^3 - \frac{1}{2} \pi x^2 + \pi^2 x - \frac{1}{3} \pi^3 , \pi \leq x < 2\pi .
\end{array} \right.
\end{equation*}
\begin{equation*}
y(0) = 0, \ y'(0) = 0
\end{equation*}
\begin{equation*}
y^\delta (x) = y(x) + \delta \frac{\sin  k x}{\sqrt{\pi}},
\end{equation*}
\begin{equation*} y'' (x) = \left\{
\begin{array}{rcl}
 \pi - x , \ 0 \leq x < \pi, \\
x- \pi , \pi  \leq x < 2\pi .
\end{array} \in H^1_{per}(0,2\pi) \right  .
\end{equation*}
\end{example}
 \subsubsection{Third order}
\begin{example}
Set
\begin{equation*}
 y(x) = \left\{
  \begin{array}{rcl}
   \frac{1}{6} \pi x^3 - \frac{1}{24} x^4 , \ 0 \leq x < \pi, \\
\frac{1}{24} x^4 - \frac{1}{6} \pi x^3 + \frac{1}{2} \pi^2 x^2 - \frac{1}{3} \pi^3 x +\frac{1}{12} \pi^4 , \pi \leq x < 2\pi .
\end{array} \right.
\end{equation*}
\begin{equation*}
y(0) = 0, \ y'(0) = 0, \ y''(0) =0;
\end{equation*}
\begin{equation*}
y^\delta (x) = y(x) + \delta \frac{\sin  k x}{\sqrt{\pi}},
\end{equation*}
\begin{equation*} y''' (x) = \left\{
\begin{array}{rcl}
 \pi - x , \ 0 \leq x < \pi, \\
x- \pi , \pi  \leq x < 2\pi .
\end{array} \in H^1_{per}(0,2\pi) \right  .
\end{equation*}
\end{example}

\begin{figure}[htbp]
\centering
 \includegraphics[width=0.8\textwidth]{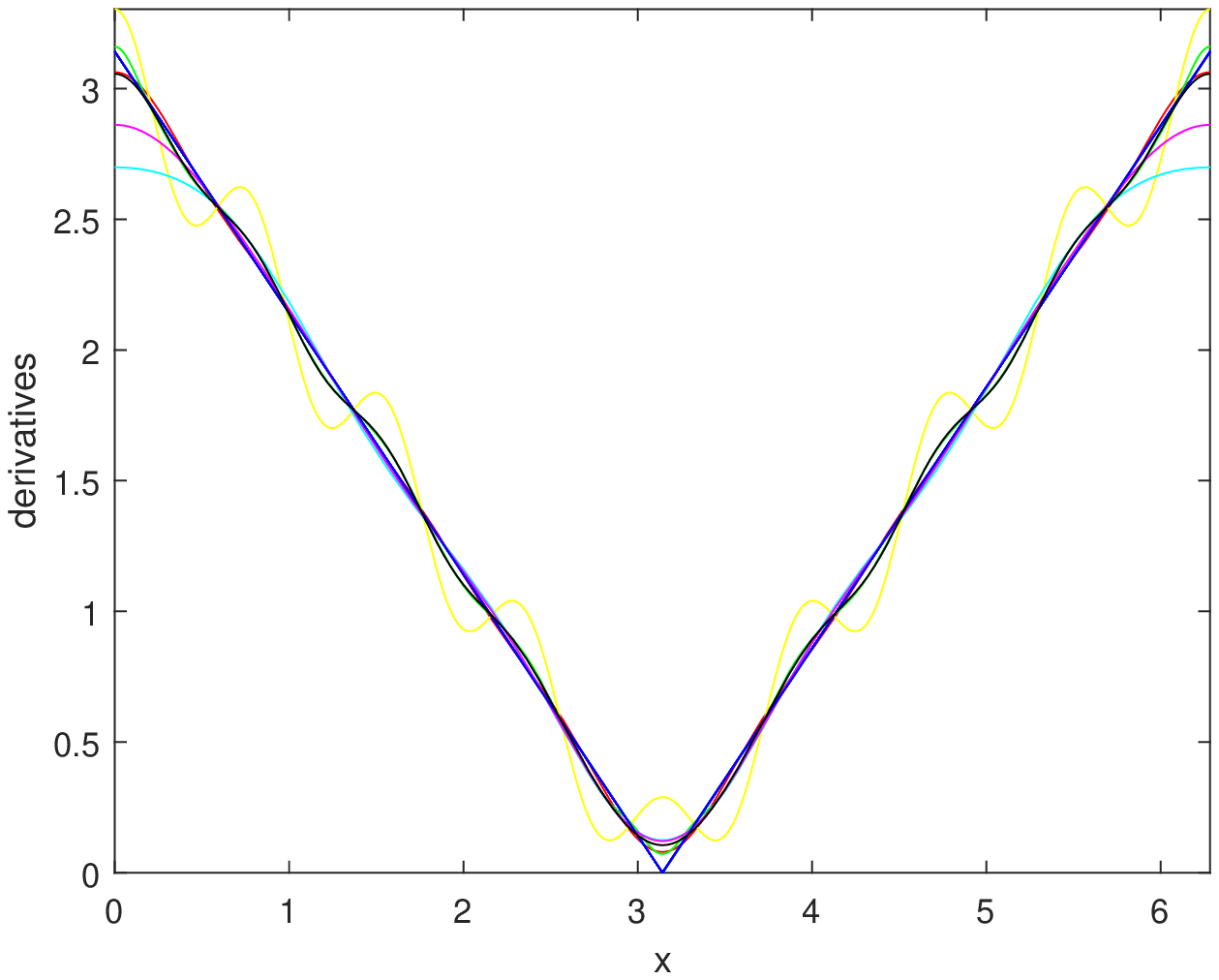}
\caption{The figure corresponds to Example 8.2 where the bule curve denotes the exact derivative, the red, yellow, green curves denote the case $ (\delta, \delta_i, n) = (0.1, 0, 7) $, $ (\delta, \delta_i, n) = (0.05, 0, 10) $ and $ (\delta, \delta_i, n) = (0.01, 0, 23) $  and the lightcyan, manganese purple, black curves denote the case $ (\delta, \delta_i, n) = (0.1, 0.1, 4) $, $ (\delta, \delta_i, n) = (0.05, 0.05, 5) $ and $ (\delta, \delta_i, n) = (0.01, 0.01, 12) $ respectively. }
\end{figure}
\begin{figure}[htbp]
\centering
 \includegraphics[width=0.8\textwidth]{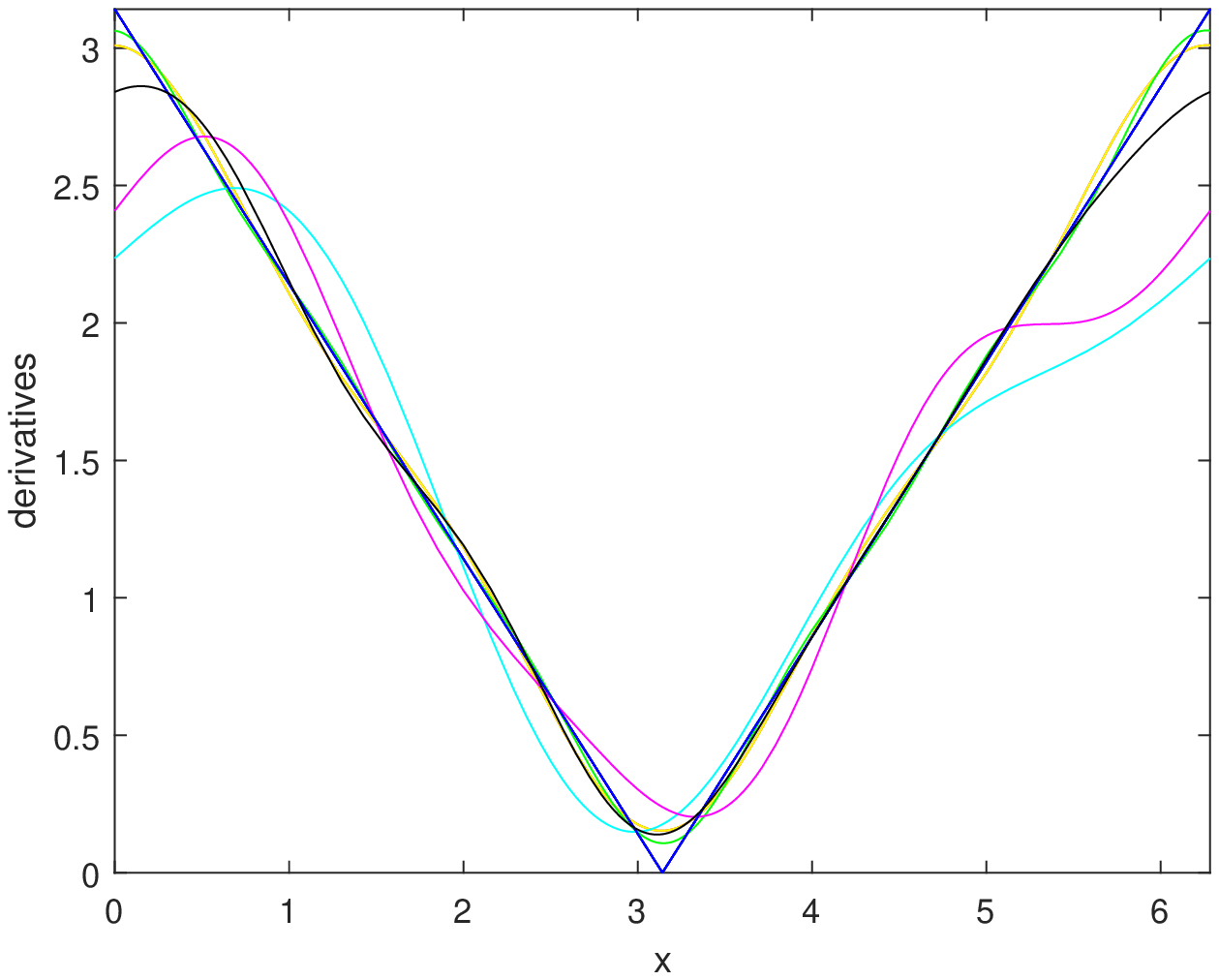}
\caption{The figure corresponds to Example 8.3 where the bule curve denotes the exact derivative, the red, yellow, green curves denote the case $ (\delta, \delta_i, n) = (0.1, 0, 3) $, $ (\delta, \delta_i, n) = (0.05, 0, 3) $ and $ (\delta, \delta_i, n) = (0.01, 0, 6) $  and the lightcyan, manganese purple, black curves denote the case $ (\delta, \delta_i, n) = (0.1, 0.1, 2) $, $ (\delta, \delta_i, n) = (0.05, 0.05, 3) $ and $ (\delta, \delta_i, n) = (0.01, 0.01, 4) $ respectively. }
\end{figure}

\begin{figure}[htbp]
\centering
 \includegraphics[width=0.8\textwidth]{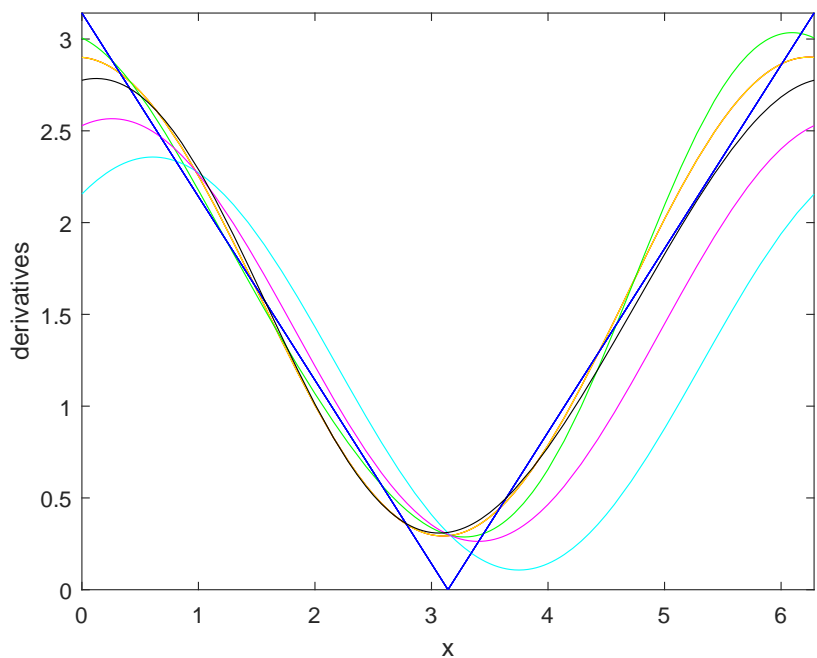}
\caption{The figure corresponds to Example 8.4 where the bule curve denotes the exact derivative, the red, yellow, green curves denote the case $ (\delta, \delta_i, n) = (0.1, 0, 1) $, $ (\delta, \delta_i, n) = (0.05, 0, 1) $ and $ (\delta, \delta_i, n) = (0.01, 0, 2) $  and the lightcyan, manganese purple, black curves denote the case $ (\delta, \delta_i, n) = (0.1, 0.1, 1) $, $ (\delta, \delta_i, n) = (0.05, 0.05, 1) $ and $ (\delta, \delta_i, n) = (0.01, 0.01, 2) $ respectively. }
\end{figure}
\subsubsection{Unified observation on cases with periodic weakly differentiable derivative}
In this subsection, we utilize strategies (6.4) and (7.1) to determine regularization parameter for two cases with different noise level respectively. The figures 1-3 show the good effectiveness of strategy proposed in this paper in a general aspect.
\subsection{On discontinuous derivatives}
 \indent In the following numerical examples with non-periodic discontinuous derivatives, we choose to adjust parameter
 \begin{equation*}
 n^{(p)} = n^{(p)}(\delta, \delta_i), \  p = 1, 2, 3,
 \end{equation*}
 with experiments, not by (6.3) for the uncertainties to determine $ \kappa $. Figures corresponding to least-error $ r $ in numerical differentiation of each order are attached.
\subsubsection{First order}
\begin{example}
Set
\begin{equation*}
 y(x) = \left\{
  \begin{array}{rcl}
   x , \ 0 \leq x < 4, \\
4 , \ 4 \leq x < 6, \\
7- \frac{x}{2} , 6 \leq x < 2\pi .
\end{array} \right.
\end{equation*}
\begin{equation*}
y(0) = 0,
\end{equation*}
\begin{equation*}
y^\delta (x) = y(x) + \delta \frac{\sin  k x}{\sqrt{\pi}},
\end{equation*}
\begin{equation*} y' (x) = \left\{
\begin{array}{rcl}
 1 , \ 0 \leq x < 4, \\
0 , \ 4 \leq x < 6, \\
-\frac{1}{2} , 6 \leq x < 2\pi .
\end{array} \right.
\end{equation*}
\end{example}

\subsubsection{Second order}
\begin{example}
Set
\begin{equation*} y(x) = \left\{
\begin{array}{rcl}
  x^3-7x^2 , \ 0 \leq x < 4, \\
  x^2 -16x , \ 4 \leq x < 6, \\
-4x-36, 6 \leq x < 2\pi .
\end{array}\right.
\end{equation*}
\begin{equation*}
y(0) = 0, y'(0) = 0,
\end{equation*}
\begin{equation*}
y^\delta (x) = y(x) + \delta \frac{\sin  k x}{\sqrt{\pi}},
 \end{equation*}
\begin{equation*} y'' (x) = \left\{
\begin{array}{rcl}
6 x - 14 , \ 0 \leq x < 4, \\
2 , \ 4 \leq x < 6, \\
 0 , 6 \leq x < 2\pi .
 \end{array}\right.
\end{equation*}

\end{example}

\subsubsection{Third order}
\begin{example}
Set
\begin{equation*} y(x) =\left\{
 \begin{array}{rcl}
   x^4 + x^3 , \ 0 \leq x < 4, \\
 13x^3 -48x^2+64x , \ 4 \leq x < 6, \\
 186x^2-1340x+2808 , 6 \leq x < 2\pi .
 \end{array} \right.
\end{equation*}
\begin{equation*}
y(0) = 0, y'(0) = 0, y''(0)=0,
\end{equation*}
and
\begin{equation*}
y^\delta (x) = y(x) + \delta \frac{\sin  k x}{\sqrt{\pi}},
\end{equation*}
\begin{equation*} y''' (x) = \left\{
\begin{array}{rcl}
 24 x+6 , \ 0 \leq x < 4, \\
78 , \ 4 \leq x < 6, \\
 0 , 6 \leq x < 2\pi .
 \end{array}\right.
\end{equation*}
\end{example}

\begin{table}
{\begin{tabular}{cccccccc}
\hline
& & $ n $ & $ 4 $ & $ 6 $ & $ 8 $ & $ 16 $& 24  \\
\cline{2-8}
 $ p=1 $& $ \delta_i = 0 $  & $ r $ & 0.2786 & 0.2551 & 0.2294 & 0.1474 & 0.1294 \\
\cline{2-8}
 & $ \delta_i = 0.01 $ & $ r $ & 0.2734 &  0.2486 & $ 0.2216 $ & $ 0.1378 $ & 0.1187 \\
\cline{2-8}
 $ p=2 $& $ \delta_i = 0 $  & $ r $ & 0.4148 & 0.3175 & 0.2754 & 0.2068 & 0.1636 \\
\cline{2-8}
 & $ \delta_i = 0.01 $ & $ r $ & 0.4239 &  0.3323 & 0.2948 & 0.2603 &  0.2667 \\
\cline{2-8}
 $ p=3 $& $ \delta_i = 0 $  & $ r $ & 0.1413 & 0.1185 & 0.1209 & 0.1137  & 0.1490 \\
\cline{2-8}
 & $ \delta_i = 0.01 $ & $ r $ & 0.1446 & 0.1185  & 0.1060 &0.1383& 0.4010 \\
\hline
\end{tabular}}
\caption{Above experiments  correspond to Example 8.2, 8.3 and 8.4 respectively. Case I,II are uniformly set as $ (\delta, \delta_i) = (0.01, 0 ) $ and $ ( \delta, \delta_i ) = ( 0.01, 0.01) $ respectively. Notice that $ r $ denotes the relative error. ($ k=8 $) }
\end{table}
\begin{table}
{\begin{tabular}{cccccccc}
\hline
& & $ n $ & $ 4 $ & $ 6 $ & $ 8 $ & $ 16 $& 24  \\
\cline{2-8}
 $ p=1 $& $ \delta_i = 0 $  & $ t $ & 1.73 & 2.67 & 3.89 & 12.47 & 103.05 \\
\cline{2-8}
 & $ \delta_i = 0.01 $ & $ t $ & 1.55 &  2.66 & 4.29 & 13.80 & 142.73 \\
\cline{2-8}
 $ p=2 $& $ \delta_i = 0 $  & $ t $ & 3.77 & 6.57 & 11.38 & 73.04 & 175.93 \\
\cline{2-8}
 & $ \delta_i = 0.01 $ & $ t $ & 3.95 &  6.46 & 11.09 & 116.55 &  230.05 \\
\cline{2-8}
 $ p=3 $& $ \delta_i = 0 $  & $ t $ & 6.10 & 13.06 & 19.47 & 127.01  & 304.70 \\
\cline{2-8}
 & $ \delta_i = 0.01 $ & $ t $ & 7.79 & 11.45  & 17.89 & 149.72 & 438.42 \\
\hline
\end{tabular}}
\caption{$ t $ denotes the CPU time $ (s) $ for the corresponding experiments in Table 3.}
\end{table}

\begin{figure}[htbp]
\centering
 \includegraphics[width=0.8\textwidth]{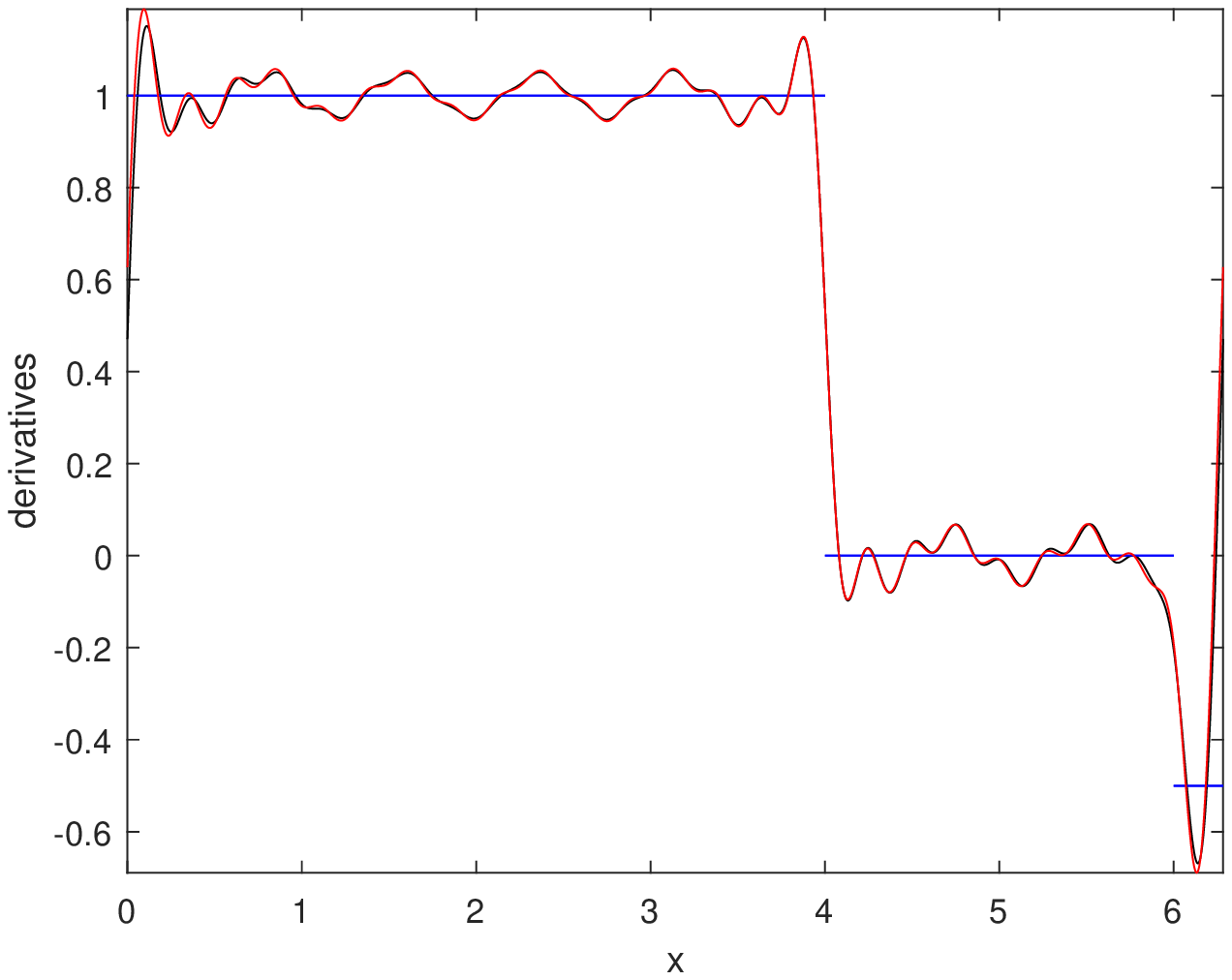}
\caption{The figure corresponds to Example 8.2 where the bule curve denotes the exact derivative, and the red, black curves denote the case $ (\delta, \delta_i, n) = (0.01, 0, 24) $ and $ (\delta, \delta_i, n) = (0.01, 0.01, 24) $ respectively. }
\end{figure}

\begin{figure}[htbp]
\centering
 \includegraphics[width=0.8\textwidth]{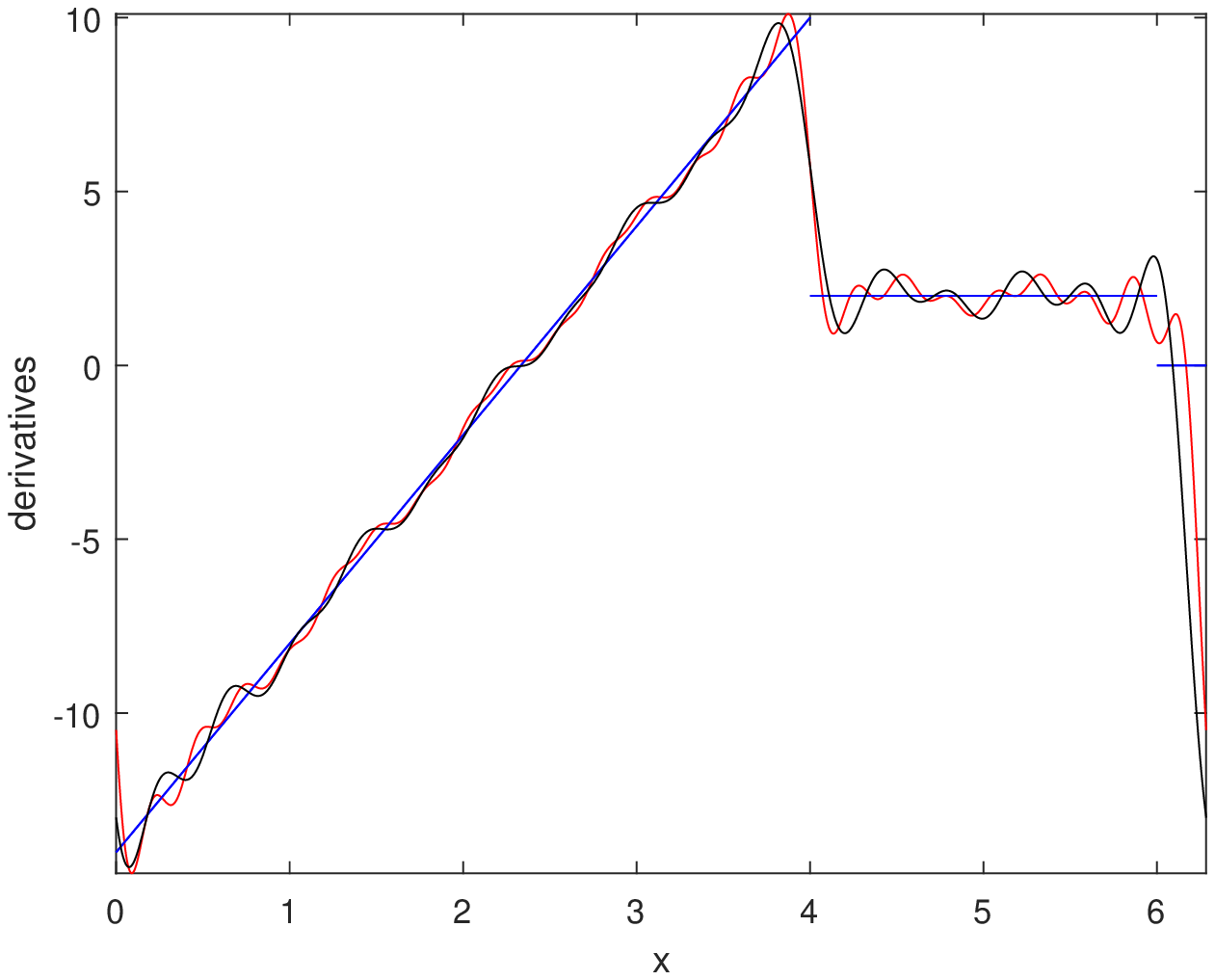}
\caption{The figure corresponds to Example 8.3 where the blue curve denotes the exact derivative, and the red, black curve denote the case $ (\delta, \delta_i, n) = (0.01, 0, 24) $ and $ (\delta, \delta_i, n) = (0.01, 0.01, 16) $ respectively.}
\end{figure}

\begin{figure}[htbp]
\centering
 \includegraphics[width=0.8\textwidth]{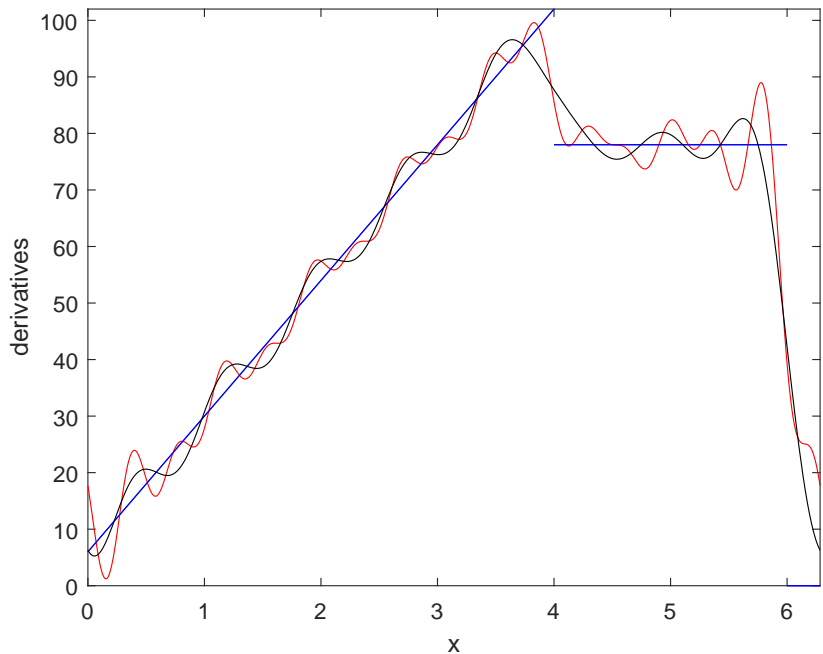}
\caption{The figure corresponds to Example 8.4 where the blue curve denote the exact derivative, and the red, black curves denote the case $ (\delta, \delta_i, n) = (0.01, 0, 16) $ and $ (\delta, \delta_i, n) = (0.01, 0.01, 8) $ respectively.}
\end{figure}

\subsubsection{Unified observation on cases with discontinuous derivative}
 It can be concluded from figure 4,5,6 that, in both cases, when regularization parameters are chosen appropriately, the computational error can be well controlled.
 However, for sake of the intersection of frequency band of $ y $ and $ \delta y $ (this does not happen in the example we list in the subsection 8.1), the good filtering effect disappears in case with discontinuous derivative.
 Besides, we note that when the choice of $ n $ increases to $ \mathcal{U}:=\{ 16, 24 \} $, the CPU time will increase to a considerable amount.

 \section{Conclusion}
The core theoretical work of this paper locates in the uniform upper estimate for
\begin{equation*}
\Vert {A^{(p)}_n}^\dagger P_n A^{(p)} \Vert_{L^2 \to L^2}
\end{equation*}
where $ A^{(p)}, A_n^{(p)} $ are defined in (2.4), (3.3) respectively. This determines the error estimate for approximation error and give a complete answer to regularization procedure.
\newline \indent In experiments, the algorithm has its advantage over other classical regularization method:
 \begin{itemize}
 \item It induces a noise-independent a-priori parameter choice strategy for function of a specific class
\begin{equation*}
y(t) =  \sum^{N_1}_{k=1} a_k \cos kt + \sum^{N_2}_{k=1} b_k \sin kt + \sum^{p-1}_{k=0} c_k t^k
\end{equation*}
where $ p $ is the order of numerical differentiation. Good filtering effect (error approaches 0) is displayed when the algorithm acts on functions of above class with best parameter choice.
\item  Derivatives discontinuities can also be recovered well although there exists a unknown constant $ \kappa $ to test in experiments.
\end{itemize}

\appendix
\section{Representation of $ M_n^{(p)}$}
\begin{equation*}
M_n^{(p)} = \left(
  \begin{array}{cccc}
    a^{(p)} & u^{(p)}_1 &\cdots & u^{(p)}_n          \\
    {v^{(p)}_1}^T & M^{(p)}_{11} & \cdots & M^{(p)}_{1n}  \\
    \vdots & \vdots & \ddots & \vdots \\
  {v^{(p)}_n}^T & M^{(p)}_{n1}  & \cdots & M^{(p)}_{nn}  \\
  \end{array}\right )_{(2n+1) \times (2n+1)}, \ p=1,2,3.
\end{equation*}
where
\begin{equation*}
a^{(1)} = \pi, a^{(2)} = \frac{2\pi^2}{3}, a^{(3)} =  \frac{\pi^3}{3}
\end{equation*}
 \begin{eqnarray*}
u^{(1)}_k = (0,\frac{\sqrt{2}}{k}), \ v^{(1)}_k = - u^{(1)}_k \\
u^{(2)}_k = (\frac{\sqrt{2}}{k^2},\frac{\sqrt{2}\pi}{k}), \  v^{(2)}_k = (\frac{\sqrt{2}}{k^2}, - \frac{\sqrt{2}\pi}{k}) \\
u^{(3)}_k = ( \frac{\sqrt{2} \pi}{k^2}, \frac{2\sqrt{2} \pi^2}{3 \cdot k}-\frac{\sqrt{2}}{k^3} ), \ v^{(3)}_k =  ( \frac{\sqrt{2} \pi}{k^2}, - \frac{2\sqrt{2} \pi^2}{3 \cdot k} + \frac{\sqrt{2}}{k^3} ), 1\leq k \leq n
\end{eqnarray*}
\begin{equation*}
M^{(1)}_{ij} = 0, \forall i \neq j,
 \end{equation*}
\begin{equation*}
M^{(1)}_{ii} = \left(
\begin{array}{cc}
0 & -\frac{1}{i} \\
\frac{1}{i} & 0 \\
\end{array}\right), \ 1\leq i \leq n.
\end{equation*}

\begin{equation*}
 M^{(2)}_{ij} = \left(
\begin{array}{cc}
0 & 0 \\
0 & -\frac{1}{i \cdot j} \\
\end{array}\right),  \  1 \leq i,j \leq n, \ i \neq j.
 \end{equation*}
\begin{equation*}
M^{(2)}_{ii} = \left(
\begin{array}{cc}
- \frac{1}{i^2} & 0 \\
0 & - \frac{3}{i^2}\\
\end{array}\right), \ 1\leq i \leq n.
\end{equation*}

\begin{equation*}
 M^{(3)}_{ij} = \left(
\begin{array}{cc}
0 & \frac{2}{ i^2 \cdot j } \\
-\frac{2}{ i \cdot j^2} & -\frac{2\pi}{ i \cdot j} \\
\end{array}\right),  \  1 \leq i,j \leq n, \ i \neq j.
 \end{equation*}
\begin{equation*}
M^{(3)}_{ii} = \left(
\begin{array}{cc}
0 & \frac{3}{i^3}  \\
 -\frac{3 }{ i^3} & -\frac{2\pi}{  i^2} \\
\end{array}\right), \ 1\leq i \leq n.
\end{equation*}

\section{Some Fourier expansions}
\begin{lemma}
For $ A^{(p)} $ defined in (2.4), when $ j \geq n+1 $, set
\begin{equation*}
 (P_n A^{(p)} (\frac{ \cos j t }{\sqrt{\pi}}))(x) =c^{(p)}_0 \frac{1}{\sqrt{2\pi}} + \sum^n_{k=1} c^{(p)}_k \frac{ \cos k x}{\sqrt{ \pi }} + \sum^n_{k=1} d^{(p)}_k \frac{ \sin k x}{\sqrt{ \pi }} ,
 \end{equation*}
Then Fourier coefficients are determined as follows:
\begin{equation}
c^{(1)}_0 = c^{(1)}_k = d^{(1)}_k = 0, \ k \in \overline{1,...,n},
\end{equation}
\begin{equation}
c^{(2)}_0=\frac{ \sqrt{2} }{ j^2},  c^{(2)}_k = 0 , d^{(2)}_k = 0, \ k \in \overline{1,...,n},
\end{equation}
\begin{equation}
c^{(3)}_0=\frac{ \sqrt{2} \pi  }{ j^2}, c^{(3)}_k = 0,  d^{(3)}_k = -\frac{2}{k j^2}, \ k \in \overline{1,...,n}.
\end{equation}
\end{lemma}
\begin{lemma}
For $ A^{(p)} $ defined in (2.4), when $ j \geq n+1 $, set
\begin{equation*}
 (P_n A^{(p)} (\frac{ \sin j t }{\sqrt{\pi}}))(x) =s^{(p)}_0 \frac{1}{\sqrt{2\pi}} + \sum^n_{k=1} s^{(p)}_k \frac{ \cos k x}{\sqrt{ \pi }} + \sum^n_{k=1} t^{(p)}_k \frac{ \sin k x}{\sqrt{ \pi }} ,
 \end{equation*}
 Then Fourier coefficients are determined as follows£º
 \begin{equation}
s^{(1)}_0 = \frac{\sqrt{2} }{ j}, \  s^{(1)}_k = 0, \  t^{(1)}_k = 0, \  k \in \overline{1,...,n},
\end{equation}
 \begin{equation}
s^{(2)}_0 = \frac{\sqrt{2} \pi }{ j}, \  s^{(2)}_k = 0, \  t^{(2)}_k = -\frac{2}{ k j},  \  k \in \overline{1,...,n},
\end{equation}
\begin{equation}
s^{(3)}_0 = \frac{2\sqrt{2} \pi^2}{3 j}-\frac{\sqrt{2}}{ j^3}, \  s^{(3)}_k =\frac{2}{ k^2 j}, \  t^{(3)}_k = -\frac{2 \pi }{ k j}  \  k \in \overline{1,...,n}.
\end{equation}
\end{lemma}
\section{Some Inequalities}
 \begin{proposition}
 Set $ L_n, K_n, F_n, T_n $ defined as in Theorem 4.1.
Then
\begin{equation}
  L^{-1}_n \in [10n, 36n ] \label{eq:(A)},
\end{equation}
\begin{equation}
K_n \in [-\frac{2}{n}, -\frac{1}{2n}],
\end{equation}
\begin{equation}
F_n \in [\frac{\sqrt{2}}{n(2n+1)},\frac{4\sqrt{2}}{n(2n+1)} ],
\end{equation}
\begin{equation}
T_n \in [ \frac{1}{396} \frac{1}{n}\frac{1}{2n+1}, \frac{3}{40} \frac{1}{n(2n+1)}] , \ n \geq 5.
\end{equation}
\end{proposition}

\section*{Acknowledgement}
I would like to thank Professor Nailin Du for his helpful discussions and suggestion that
 extend numerical differentiation into higher order. And the author thank
 Professor Hua Xiang for his comments. Besides, I really appreciate the help in figure modification from  Nianci Wu.
 The work of Yidong Luo is supported by the National Natural Science
 Foundation of China under grants 11571265 and NSFC-RGC No. 11661161017.
\section*{References}

\bibliographystyle{elsarticle-num-names.bst}

\end{document}